\documentclass{mfatshort}

\newtheorem{theorem}{Theorem}[section]

\newtheorem{proposition}[theorem]{Proposition}
\theoremstyle{definition}
\newtheorem{definition}[theorem]{Definition}
\theoremstyle{remark}
\newtheorem{remark}[theorem]{Remark}
\numberwithin{equation}{section}
\renewenvironment {proof} {\begin{trivlist} \item[\hspace{\labelsep}%
\sc Proof.]}{$\Box$ \end{trivlist}}

   \def\cN{{\mathcal N}}

\renewcommand {\phi}{\varphi}          

\newcommand {\supp}{\mathop {\rm supp}}     

\begin{document}

\title[Multidimensional moment problem  and diagonal Schur algorithm]
{Multidimensional moment problem  and diagonal Schur algorithm}

\author[I. Kovalyov]{Ivan Kovalyov}
\address{Universität Osnabrück\\
Albrechtstr. 28a  \\
49076 Osnabrück \\
Germany.}
\email{i.m.kovalyov@gmail.com}

\thanks{I. Kovalyov also gratefully acknowledges financial support by the German Research Foundation (DFG, grant 520952250).}


\begin{abstract}
The multidimensional moment problem is studied in terms of the Steiltjes transform. The diagonal step-by-step algorithm is constructed for the  multidimensional moment problem. The set of solutions of the full multidimensional moment problem is found in terms of the continued fractions.  Moreover, the diagonal step-by-step algorithm  can be applied to the special truncated multidimensional moment problem.

\end{abstract}

\maketitle

\tableofcontents

\section{Introduction}

A  Stieltjes  moment problem was studied in~\cite{Akh,St89}.
Given a sequence of  real numbers $\textbf{s}=\{s_{i}\}_{i=0}^{\infty}$, find a
positive Borel measure $\sigma$ with a support on $\mathbb{R}_{+}$, such
that
\begin{equation}\label{int1}
    \int_{\mathbb{R}_{+}}t^jd\sigma(t)=s_{j},\quad j\in\mathbb{Z}_{+}=\mathbb{N}\cup\{0\}.
\end{equation}
The problem \eqref{int1} with a finite  sequence
 $\mathbf{s}=\{s_{i}\}_{i=0}^{\ell}$
is called a truncated Stieltjes  moment problem, otherwise it is called a full Stieltjes  moment problem.

By the Hamburger--Nevanlinna theorem~\cite{Akh} the truncated Stieltjes  moment problem can be reformulated in term
of the Stieltjes transform
\begin{equation}\label{int3}
    f(z)=\int_{\mathbb{R}_{+}}\frac{d\sigma(t)}{t-z}\qquad z\in \mathbb{C}\backslash \mathbb{R}_{+}
\end{equation}
of $\sigma$ as the following interpolation problem at $\infty$
\begin{equation}\label{14.int.th_1}
	 f(z)=-\frac{s_{0}}{z}-\frac{s_{1}}{z^2}-\cdots-\frac{s_{\ell}}{z^{\ell+1}}
    +o\left(\frac{1}{z^{\ell+1}}\right),\quad\quad z\widehat{\rightarrow}\infty,
\end{equation}
where the notation $z\widehat{\rightarrow}\infty$ means that $z\rightarrow\infty$  nontangentially, that is inside the
sector $\varepsilon<\arg z<\pi-\varepsilon$ for some
$\varepsilon>0$.

\cite{D97} Recall that the set $\mathcal{N}(\textbf{s})=\{n_{j}\}_{j=1}^{N}$ of
normal indices of $\textbf{s}=\{s_j\}_{j=0}^{\ell}$ is defined by
\begin{equation}\label{3p.2.4}
\mathcal{N}(\textbf{s})=\{n_{j}: D_{n_{j}}\neq0, j=1,2,\ldots,
N\},\quad D_{n_{j}}:=\textup{det}({s}_{i+k})_{i,k=0}^{n_{j}-1}.
\end{equation}
Let us set $D_{n}^{+}:=\textup{det}(s_{i+j+1})_{i,j=0}^{n-1}$. By
the Sylvester identity (see~\cite[Proposition~3.1]{DK15} or
\cite[Lemma~5.1]{D97} for detail), the set $\mathcal{N}(\textbf{s})$
is the union of two  subsets
\begin{equation}\label{eq:3pN=mu+nu}
    \mathcal{N}(\textbf{s})= \{\nu_{j}\}_{j=1}^{N_{1}}\cup \{\mu_{j}\}_{j=1}^{N_{2}},
\end{equation}
which are selected by
 \begin{equation}\label{eq:3p.nu}
    D_{\nu_{j}}\neq0\quad\mbox{and}\quad
    D_{\nu_{j}-1}^{+}\neq0,\quad\mbox{ for all
    }j=\overline{1,N_{1}}
\end{equation}
and
\begin{equation}\label{eq:3p.mu}
     D_{\mu_{j}}\neq0\quad\mbox{and}\quad
    D_{\mu_{j}}^{+}\neq0,\quad\mbox{ for all
    }j=\overline{1,N_{2}}.
\end{equation}
Moreover, the normal indices
$\nu_{j}$ and  $\mu_{j}$ satisfy the following inequalities
\begin{equation}\label{eq:3p.nu<mu}
     0<\nu_{1}\leq\mu_{1}<\nu_{2}\leq\mu_{2}<\ldots
\end{equation}

Let us formulate a \textbf{moment problem $MP(\mathbf{s},\ell)$}:
Given a sequence of real numbers $\mathbf{s}=\{s_{j}\}_{j=0}^{\ell}$. Find the set of function $F$, which admit the  asymptotic expansion \eqref{14.int.th_1}.

Recall the result describing the basic steps of the Schur algorithm for the one-dimensional moment problem $MP(\mathbf{s},\ell)$.
\begin{theorem}\label{14p.int.th_ind} (\cite{Der03}) Let $\mathbf{s}=\{s_{j}\}_{j=0}^{\ell}$ be a sequence of real numbers and assume that $\ell>2n_1-1$, where $n_1$ is the first normal index of the sequence $\mathbf{s}$. If  $f$ admits the asymptotic expansion \eqref{14.int.th_1}, 
then $f$ cam be rewritten as
\begin{equation}\label{14.int.th_2}
	f(z)=-\cfrac{b_0}{a_0(z)+f_1(z)},
\end{equation}
where 
\begin{equation}\label{14.int.th_3}
	b_0=s_{n_1-1}, \, a_0(z)=\cfrac{1}{D_{n_1}}\begin{vmatrix}
s_{0}& s_{1} &\ldots& s_{n_1}\\
\vdots& \ddots &\ddots &\vdots\\
s_{n_1-1}& \ldots   &\ldots & s_{2n_1-1}\\
1&z&\ldots& z^{n_1}
\end{vmatrix}
	\,\mbox{and}\, f_1(z)=-\sum_{j=0}^{\ell-2n_1}\cfrac{s_j^{(1)}}{z^j},
\end{equation}
the recursive sequence $\mathbf{s}^{(1)}=\left\{s_j^{(1)}\right\}_{j=0}^{\ell-2}$ is defined by
 \begin{equation}\label{14.int.th_4}
s_{j}^{(1)}=\cfrac{(-1)^{j+n_1}}{s_0^{j+n_1+2}}\begin{vmatrix}
s_{n_1}& s_{n_1-1} &0&\ldots& 0\\
\vdots& \ddots &\ddots&\ddots & \vdots\\
\vdots&  &\ddots&\ddots &0\\
\vdots&  & &\ddots & s_{n_1-1}\\
s_{j+2n_1}& \ldots &\ldots&\ldots&s_{n_1}
\end{vmatrix},\quad j=\overline{1,\ell-2n_1}.
\end{equation}

In \cite{DK15, DK17, K17}, $MP(\mathbf{s},\ell)$ are studied in the generalized Stieltjes classes, any solutions of $MP(\mathbf{s},\ell)$ can be represented in terms of $S$-fractions. First of all, we recall the results of the odd problem:

\end{theorem}
 \begin{theorem}\label{2p.pr.alg1}
Let ${\mathbf
s}=\{s_i\}_{i=0}^{2\nu_N-2}$ be the  sequence of real numbers  and  let
$\cN({\mathbf s})=\{\nu_j\}_{j=1}^N\cup\{\mu_j\}_{j=1}^{N-1}$ be the set of normal indices of the sequence ${\mathbf
s}$. Then any solution of the moment problem $MP(\mathbf{s}, 2\nu_N-2)$
 admits the following representation 
 \begin{equation}\label{5p.eq:3.6*9}
    f(z)= \frac{1}{\displaystyle -z m_{1}(z)+\frac{1}{\displaystyle
    l_{1}(z)+\frac{1}{\displaystyle -zm_{2}(z)+\cdots+\frac{1}{-zm_{N}(z)+\frac{\displaystyle 1}{\displaystyle\tau(z)}}
    }}},
\end{equation}
 where the parameter $\tau$ satisfies the following
 \begin{equation}\label{eq:tau_j}
    \frac{\displaystyle1}{\displaystyle\tau(z)}=o(z),\quad
    z\widehat{\rightarrow}\infty.
\end{equation}
and atoms $(m_j, l_j)$ can be calculated by 
  \begin{equation}\label{eqimnm1:}
        m_j(z)=\frac{(-1)^{\nu+1}}{D^{(2j-1)}_\nu}
        \begin{vmatrix}
            0 & \ldots & 0 &  \mathfrak s^{(2j-1)}_{\nu-1} &  \mathfrak s^{(2j-1)}_{\nu} \\
            \vdots &  & \ldots & \ldots & \vdots \\
            \mathfrak s^{(2j-1)}_{\nu-1} &\mathfrak s^{(2j-1)}_{\nu} & \ldots & \ldots & \mathfrak s^{(2j-1)}_{2\nu_{}-2} \\
            1 & z & \ldots& z^{\nu-2}& z^{\nu-1} \\
        \end{vmatrix},
\end{equation}
\begin{equation}\label{eq:l_j0xxx1}
l_j(z)
   = \left\{
   \begin{array}{cl}
       \frac{1}{{\mathfrak{s}}_{-1}^{(2j)}}=(-1)^{\nu
    +1} \mathfrak s_{\nu-1}^{(2j-1)}\frac{D_{\nu}^{(2j-1)}}{D_{\nu}^{(2j-1)+}},&       \mbox{ if }\nu_j=\mu_j;\\
       \frac{1}{{ \mathfrak{s}}^{(2j)}_{\mu-1}{ D}^{(2j)}_{\mu}}
    \begin{vmatrix}
    { \mathfrak{s}}^{(2j)}_{0} & \ldots & { \mathfrak{s}}^{(2j)}_{\mu-1} & { \mathfrak{s}}^{(2j)}_{\mu} \\
    \cdots & \cdots & \cdots & \cdots \\
    { \mathfrak{s}}^{(2j)}_{\mu-1} & \ldots & { \mathfrak{s}}^{(2j)}_{2\mu-2} & { \mathfrak{s}}^{(2j)}_{2\mu-1} \\
    1 & \ldots &z^{\mu-1} &z^{\mu} \\
    \end{vmatrix},& \mbox{if}\quad \nu_j<\mu_j .\\
    \end{array}
    \right.,
\end{equation}
 where the recurrence sequence $\{{s}^{(2j)}_{i}\}_{i=-1}^{2(\nu_N-\nu_j)-2}$ and
$\{s^{(2j-1)}_{i}\}_{i=0}^{2(\nu_N-\mu_j)-2}$ are found by  
\[
    T(m_{\nu_j-1}^{(j)},\ldots,m_0^{(j)},-{\mathfrak s}^{(2j)}_{-1},\ldots,-
 {\mathfrak s}^{(2j)}_{\ell_j-2\nu_j}) \, T(\mathfrak s^{(2j-1)}_{\nu_j-1},\dots,\mathfrak s^{(2j-1)}_{\ell_{j}}) =
 I_{\ell_j-\nu_1+2},\quad \ell_j=\ell-2\mu_{j-1},
\]
\[
     T(l_{\mu_j-\nu_j}^{(j)},\ldots,l_0^{(j)},-{\mathfrak s}^{(2j+1)}_{0},\ldots,-
 {\mathfrak s}^{(2j+1)}_{\ell-2\mu_j}) \, T({\mathfrak s}^{(2j)}_{\mu_j-\nu_j-1},\dots,{\mathfrak {s}}^{(2j)}_{\ell_j-2\nu_j}) = I_{\ell-\mu_j-\nu_j+2},
\]
$\nu=\nu_j-\mu_{j-1}$, $\mu=\mu_{j}-\nu_j$  and $\mathfrak s_i^{ (1)}=\mathfrak s_i$.
 \end{theorem}

Next, we consider the even case:

\begin{theorem}\label{14p.int.th_ind**}  Let ${\mathbf
s}=\{s_i\}_{i=0}^{2\mu_N-1}$ be the  sequence of real numbers  and  let
$\cN({\mathbf s})=\{\nu_j\}_{j=1}^N\cup\{\mu_j\}_{j=1}^{N}$ be the set of normal indices of the sequence ${\mathbf
s}$. Then any solution of the moment problem $MP(\mathbf{s}, 2\mu_N-1)$
 admits the following representation 
 \begin{equation}\label{5p.eq:f1}
    f(z)= \frac{1}{\displaystyle -z m_{1}(z)+\frac{1}{\displaystyle
    l_{1}(z)+\cdots+\frac{1}{-zm_{N}(z)+\displaystyle\frac{1}{\displaystyle
    l_{N}(z)+\tau(z)}    }}},
\end{equation}
where the parameter $\tau$ satisfies 
 \begin{equation}\label{eq:tau_j009}
    \tau(z)=o(1),\quad  z\widehat{\rightarrow}\infty,
\end{equation}
the atoms $(m_j, l_j)$ can befound by \eqref{eqimnm1:}--\eqref{eq:l_j0xxx1}.
\end{theorem}

\begin{definition} \label{def:3p.4.3}(\cite{DK17})
Let ${\mathbf s}=\{s_i\}_{i=0}^\ell$ and let $\mathcal{N}({\mathbf s})=\{n_j\}_{j=1}^N$.
A sequence ${\mathbf s}$ is called  regular,
if one of the following equivalent conditions  holds:
\begin{enumerate}
  \item [(1)] $P_{{n}_{j}}(0)\neq0$ for every $j\le N$;
  \item [(2)] $D_{{n}_{j}-1}^{+}\neq0$ for every $j\le N$;
  \item [(3)] $D_{{n}_{j}}^{+}\neq0$ for every $j\le N$;
  \item [(4)] $\nu_{j }=\mu_{j}$   for all $j$,  such that $\nu_j,\mu_j\in\mathcal{N}({\mathbf s})$.
\end{enumerate}
\end{definition}
\begin{remark} If the sequence ${\mathbf s}=\{s_i\}_{i=0}^\ell$ is the regular (we mean that $\ell=2n_N-2$ or $\ell=2n_N-1$,  $n_N\in \mathcal{N}({\mathbf s})$), then all $l_j=$const in 
\eqref{5p.eq:3.6*9}  (\eqref{5p.eq:f1}).
\end{remark}

\subsection{Multidimensional moment problem} Let $\mu$ is a nonnegative Borel measure on $\mathbb{R}_{+}^n$, where $\supp(\mu)=A\subseteq \mathbb{R}^n$. The moment sequence 
 $\mathbf{s}=\{s_{i,j}\}_{i,j=0}^\ell$  is defined by
\begin{equation}\label{14.2.1}
s_{\alpha_1,...,\alpha_n}=\int_A t_1^{\alpha_1}\times\ldots\times  t_n^{\alpha_n}d\mu(t_1,...,t_n).
\end{equation}

However, we can reformulate the moment problem \eqref{14.2.1}  in terms of the generalized moments. Define a linear functional $\mathfrak{S}$ by
\begin{equation}\label{14.2.3S}
	\mathfrak{S}(t_1^{\alpha_1}\times\ldots\times  t_n^{\alpha_n})=s_{\alpha_1,...,\alpha_n}
\end{equation}

In the present paper,  we study a multidimensional problem in a similar way as the one--dimensional problem \eqref{14.int.th_1}, i.e. we mean to investigate it as the interpolation problem. Recall  a $n-$variate Stieltjes transform (see \cite{Cuyt2004})
\begin{equation}\label{20.int. st tr.1}
\int_{\mathbb{R}_+^n}\cfrac{d\mu(t_1,t_2,..., t_n)}{1+ \sum\limits_{i=1}^n t_i\tilde{z}_i}.
\end{equation}

Setting $\tilde z_i=-\cfrac{1}{z_i}$, we  obtain  the  associated function $F$ of the multidimensional problem  in  the following form
\begin{equation}\label{20.2.2}
 F(z_1,z_2,...,z_n)=-\frac{1}{\prod\limits_{i=1}^n z_i}\int_{\mathbb{R}_+^n}\cfrac{d\mu(t_1,t_2,..., t_n)}{1- \sum\limits_{i=1}^n \cfrac{t_i}{z_i}},
\quad \mbox{where}\, z_i\widehat{\rightarrow}\infty.
\end{equation}

Moreover, \eqref{20.2.2} can be rewritten  in terms of the Stieltjes transform as
\begin{equation}\label{20.2.3}\begin{split}
    F(z_1,z_2,...,z_n)&=-\frac{1}{\prod\limits_{i=1}^n z_i}\int_{\mathbb{R}_+^n}\cfrac{d\mu(t_1,..., t_n)}{1- \sum\limits_{i=1}^n \cfrac{t_i}{z_i}}=
    \\&=-\frac{1}{\prod\limits_{i=1}^n z_i} \int_{\mathbb{R}_+^n}\sum_{k=0}^\infty\left( \sum\limits_{i=1}^n \cfrac{t_i}{z_i}\right)^k d\mu(t_1,..., t_n)=\\&=
    -\int_{\mathbb{R}_+^n} \sum_{k=0}^\infty \sum_{\begin{matrix} 
  \alpha_i^k\geq0  \\
  \alpha_1^k+\ldots \alpha_n^k=k \\
   \end{matrix} } \binom{k}{\alpha_1^k, \ldots, \alpha_n^k}\cfrac{\prod\limits_{j=1}^n t_j^{\alpha_j^k}d\mu(t_1,..., t_n)}{\prod\limits_{j=1}^n z_j^{\alpha_j^k+1}}=\\&=
    -\sum_{k=0}^\infty \sum_{\begin{matrix} 
  \alpha_i^k\geq0  \\
  \alpha_1^k+\ldots \alpha_n^k=k \\
   \end{matrix} }\binom{k}{\alpha_1^k, \ldots, \alpha_n^k}\cfrac{s_{\alpha_1^k,\alpha_2^k,\ldots, \alpha_n^k}}{\prod\limits_{j=1}^n z_j^{\alpha_j^k+1}}.
\end{split}
\end{equation}

 \textbf{ The multidimensional  problem} $\bold{MP}(\mathbf{s},\ell)$ can be formulated as:
 
 Given a sequence of real numbers $\mathbf{s}=\{s_{i_1,i_2,...,i_n}\}_{i_1,...,i_n=0}^\ell$. Describe the set of function $F$, such that
\begin{equation}\label{20.2.3wq}\begin{split}
 	 F(z_1,z_2,...,z_n)&=  -\sum_{k=0}^{\ell} \sum_{\begin{matrix} 
  \alpha_i^k\geq0  \\
  \alpha_1^k+\ldots \alpha_n^k=k \\
   \end{matrix} }\binom{k}{\alpha_1^k, \ldots, \alpha_n^k}\cfrac{s_{\alpha_1^k,\alpha_2^k,\ldots, \alpha_n^k}}{z_1^{\alpha_1^k+1}z_2^{\alpha_2^k+1}\ldots z_n^{\alpha_n^k+1}}+\\&
   +o\!\left( \!\!\sum_{\begin{matrix} 
  \alpha_i^{\ell}\geq0  \\
  \alpha_1^{\ell}\!+\!\ldots\!+\! \alpha_n^{\ell}\!=\!{\ell} \\
   \end{matrix}}\cfrac{1}{z_1^{\alpha_1^{\ell}+1}\ldots z_n^{\alpha_n^{\ell}+1}}\right).
\end{split} \end{equation}
 
 If $\ell$ is finite, the $\bold{MP}(\mathbf{s},\ell)$ is called a truncated problem, otherwise one is called a full problem.

 In Section 2, we discuss  a simple case, when  all moments not belonging to the main diagonal of the Hankel matrix are equal to zero. The obtained results are basic for the general case, we are going to use one for the next sections. In section 3, we study the case, when  all moments not belonging to a fixed subdiagonal diagonal of the Hankel matrix are equal to zero. Finally, in Section 4, we consider the full problem and describe the set of solutions. The description of solutions of the full problem is based on the results of the previous sections. 
 
\section{Diagonal type}
In this section, we study a simple  case of the multidimensional problem:

Let  $\mathbf{s}=\{s_{i_1,i_2,...,i_n}\}_{i_1,...,i_n=0}^\ell$ be a sequence of real numbers, such that  
\begin{equation}\label{20.3.1}
 s_{i_1,i_2,...,i_n}=0\quad\mbox{for}\quad i_j\neq i_k.
\end{equation}
Then the associated function is denoted by $F_{d_0}$ and one takes the following form
\begin{equation}\label{20.3.2}
	F_{d_0}(z_1,...,z_n)=-\cfrac{s_{0,...,0}}{\prod\limits_{i=1}^n z_i}-\cfrac{\binom{n}{1,...,1}s_{1,...,1}}{\prod\limits_{i=1}^n z_i^{2}}-\ldots-
	\cfrac{\binom{n\times \ell}{\ell,...,\ell}s_{\ell,...,\ell}}{\prod\limits_{i=1}^n z_i^{\ell+1}}+o\left(\cfrac{1}{\prod\limits_{i=1}^n z_i^{\ell+1}}\right).
\end{equation}
Setting 
\begin{equation}\label{20.3.3}
	\mathfrak s_j^{(d_0)}=\binom{n\times j}{j,...,j}s_{j,...,j}\quad\mbox{and}\quad z=\prod\limits_{i=1}^n z_i,
\end{equation}
the representation  \eqref{20.3.2} can be rewritten  as
\begin{equation}\label{20.3.4}
	F_{d_0}(z)=-\cfrac{\mathfrak s^{(d_0)}_0}{z}-\cfrac{\mathfrak s^{(d_0)}_1}{z^2}-\ldots-\cfrac{\mathfrak s^{(d_0)}_{\ell}}{z^{\ell+1}}+o\left(
	\cfrac{1}{z^{\ell+1}}\right).
\end{equation}
Moreover, $\mathfrak s^{(d_0)}=\left\{\mathfrak s^{(d_0)}_i\right\}_{i=0}^{\ell}$ is called  $d_0$-associated sequence of $\mathbf{s}$. Hence, we obtain one-dimensional problem.

 \textbf{$d_0$-diagonal truncated problem} $\bold{MP}(\mathbf{s},\ell)$ is formulated as:

Given a sequence of real numbers  $\mathbf{s}=\{s_{i_1,i_2,...,i_n}\}_{i_1,...,i_n=0}^\ell$, such that \eqref{20.3.1} holds. Describe the set of functions $F_{d_0}$, such that have the asymptotic expansion \eqref{20.3.2}.

\begin{proposition}\label{20.th3.1} Let the sequence of real numbers  $\mathbf{s}=\{s_{i_1,i_2,...,i_n}\}_{i_1,...,i_n=0}^{2n^{(d_0)}_N-1}$ satisfy \eqref{20.3.1}. Let  $\mathfrak s^{(d_0)}=\{\mathfrak s^{(d_0)}_i\}_{i=0}^{2n_j-1}$ be  $d_0$-associated sequence of $\mathbf{s}$ and   let $\mathcal{N}({\mathfrak s^{(d_0)}})=\{n^{(d_0)}_j\}_{j=1}^N$ be the set of normal indices  of $\mathfrak s^{(d_0)}$. Then any solution of $d_0$-diagonal truncated problem $\bold{MP}(\mathbf{s},2n^{(d_0)}_N-1)$ admits the following representation
\begin{equation}\label{20.3.5}
	F_{d_0}(z_1,...z_n)=-\cfrac{b^{(d_0)}_0}{a^{(d_0)}_{0}(z_1,...,z_n)-\cfrac{b_1^{(d_0)}}{a_1^{(d_0)}(z_1,...,z_n)-\ldots -\cfrac{b_{N-1}^{(d_0)}}{a_{N-1}^{(d_0)}(z_1,...,z_n)}+\tau(z_1,...,z_n)}},
\end{equation}
where the parameter $\tau$ satisfies the condition
\begin{equation}\label{20.3.6}
	\tau(z_1,...,z_n)=o(1)
\end{equation}
and atoms $(a^{(d_0)}_i,b^{(d_0)}_i)$ can be found by
 \begin{equation}\label{20.3.7}\begin{split}
	&b^{(d_0)}_0=\mathfrak s^{(d_0)}_{n_1-1}\quad  \mbox{and}\quad b^{(d_0)}_j=\mathfrak s^{(d_0, j)}_{n^{d_0}_{j}-n^{(d_0)}_{j-1}-1}\\&
	 a^{(d_0)}_{j}(z_1,...,z_n)=\cfrac{1}{D^{(d_0)}_{\nu}} \begin{vmatrix}
	\mathfrak  s^{(d_0, j-1)}_{0}&\mathfrak  s^{(d_0, j-1)}_{1}& \ldots&  \mathfrak s^{(d_0, j-1)}_{\nu}\\
	  \ldots& \ldots&\ldots& \ldots\\
	   \mathfrak s^{(d_0, j-1)}_{\nu-1}& \mathfrak s^{(d_0, j-1)}_{\nu}& \ldots&\mathfrak  s^{(d_0, j-1)}_{2\nu-1}\\
	    1& \prod\limits_{i=1}^n z_i&\ldots& \prod\limits_{i=1}^n z^{\nu}_i
	 \end{vmatrix}
	 ,\\&
	\mathfrak s^{(d_0, j)}_i=\cfrac{(-1)^{i+\nu}}{\left(\mathfrak s^{(d_0, j-1)}_{\nu-1}\right)^{i+\nu+2}}\begin{vmatrix}
\mathfrak s^{(d_0, j-1)}_{\nu}& \mathfrak s^{(d_0, j-1)}_{\nu-1} &0&\ldots& 0\\
\vdots& \ddots &\ddots&\ddots & \vdots\\
\vdots&  &\ddots&\ddots &0\\
\vdots&  & &\ddots &\mathfrak  s^{(d_0, j-1)}_{\nu-1}\\
\mathfrak s^{(d_0, j-1)}_{2\nu+i}& \ldots &\ldots&\ldots& \mathfrak s^{(d_0, j-1)}_{\nu-1}
\end{vmatrix},
\end{split}
\end{equation}
where $i=\overline{0,2n^{(d_0)}_N-2n^{(d_0)}_j-1}$, $j=\overline {0, N}$, $\mathfrak s^{(d_0, 0)}_{\nu}=\mathfrak s^{(d_0)}_{\nu}$, $\nu=n^{(d_0)}_j-n^{(d_0)}_{j-1}$ and $n^{(d_0)}_{0}=0$.
\end{proposition}
\begin{proof} Assume the sequence of real numbers $\mathbf{s}=\{s_{i_1,i_2,...,i_n}\}_{i_1,...,i_n=0}^{2n^{(d_0)}_N-1}$ satisfies \eqref{20.3.1}. Let $\mathfrak s^{(d_0)}=\{\mathfrak s^{(d_0)}_i\}_{i=0}^{2n^{(d_0)}_N-1}$ be  $d_0$-associated sequence of $\mathbf{s}$ and   let $\mathcal{N}({\mathfrak s^{(d_0)}})$ be the set of normal  indices of the  associated sequence $\mathfrak s^{(d_0)}$. Then any solution of  $d_0$-diagonal truncated problem $\bold{MP}(\mathbf{s},2n^{(d_0)}_N-1)$ takes the following asymptotic expression
\begin{equation}\label{20.3.4}
	F_{d_0}(z)=-\cfrac{\mathfrak s^{(d_0)}_0}{z}-\cfrac{\mathfrak s^{(d_0)}_1}{z^2}-\ldots-\cfrac{\mathfrak s^{(d_0)}_{2n_N-1}}{z^{2n_N}}+o\left(
	\cfrac{1}{z^{2n_N}}\right).
\end{equation}
By Theorem \ref{14p.int.th_ind}, we obtain \eqref{20.3.5}--\eqref{20.3.7}.  This completes the proof.~\end{proof}

Next, we describe the set of solutions of $d_0$-diagonal truncated problem $\bold{MP}(\mathbf{s},\ell)$ in terms of $S$--fractions.
\begin{proposition}\label{20.th3.3} Let $\mathbf{s}=\{s_{i_1,i_2,...,i_n}\}_{i_1,...,i_n=0}^{2\nu_N^{(d_0)}-2}$ be a sequence of real numbers such that  \eqref{20.3.1} holds. Let  $\mathfrak s^{(d_0)}=\{\mathfrak s^{(d_0)}_i\}_{i=0}^{2\nu_N^{(d_0)}-2}$ be  $d_0$-associated sequence of $\mathbf{s}$ and let $\mathcal{N}({\mathfrak s^{(d_0)}})=\{\nu_j^{(d_0)}\}_{j=1}^N\cup\{\mu_j^{(d_0)}\}_{j=1}^{N-1}$ be the  set of normal indices of $\mathfrak s^{(d_0)}$. Then any solution of $d_0$--diagonal truncated problem $\bold{MP}(\mathbf{s},2\nu^{(d_0)}_N-2)$ admits the following representation
 \begin{equation}\label{20.3.8}
 \begin{split}
 	F_{d_0}(z_1,..., z_n)&=\cfrac{1|}{\big|-m_1^{(d_0)}(z_1,..., z_n)\prod\limits_{j=1}^nz_j}+\cfrac{1|}{\big|l_1^{(d_0)}(z_1,...,z_n)}+\ldots+\\&+
	\cfrac{1|}{\big|-m_N^{(d_0)}(z_1,..., z_n)\prod\limits_{j=1}^nz_j+\cfrac{1}{\tau(z_1,...,z_n)}},
 \end{split}
  \end{equation}
 where  the parameter $\tau$ satisfy the following 
   \begin{equation}\label{20.3.9}
  \cfrac{1}{\tau(z_1, ..., z_n)}=O\left(1\right),
   \end{equation}
  the atoms $\left(m_j^{(d_0)} , l^{(d_0)}_j\right)$ can be calculated by
 \begin{equation}\label{eq:}
        m_j^{(d_0)}(z_1,..., z_n)\!=\!\frac{(-1)^{\nu^{(d_0)}\!+\!1}}{D^{(d_0,2j\!-\!1)}_{\nu^{(d_0)}}}\!
        \begin{vmatrix}
            0 & \ldots & 0 &  \mathfrak s^{(d_0,2j\!-\!1)}_{\nu^{(d_0)}\!-\!1} &  \mathfrak s^{(d_0,2j\!-\!1)}_{\nu^{(d_0)}} \\
            \vdots &  & \ldots & \ldots & \vdots \\
            \mathfrak s^{(d_0,2j\!-\!1)}_{\nu^{(d_0)}\!-\!1} &\mathfrak s^{(d_0,2j\!-\!1)}_{\nu^{(d_0)}} & \ldots & \ldots & \mathfrak s^{(d_0,2j\!-\!1)}_{2\nu^{(d_0)}\!-\!2} \\
            1 \!&\! \prod\limits_{j=1}^n\!\!z_j \!&\! \ldots\!&\! \prod\limits_{j=1}^n\!\!z_j^{\nu^{(d_0)}\!-\!2}\!& \prod\limits_{j=1}^n\!\!z_j^{\nu^{(d_0)}\!-\!1}\! \\
        \end{vmatrix},
\end{equation}

\begin{equation}\label{20.3.10}
l_j^{(d_0)}\!(z_1,\!..., z_n)
  \! =\! \!\left\{\!\!\!\!
   \begin{array}{cl}
       \frac{1}{{\mathfrak{s}}_{\!-1}^{(d_0,2j)}}=(-1)^{\nu^{(d_0)}
    +1} \mathfrak s_{\nu^{(d_0)}\!-\!1}^{(d_0,2j\!-\!1)}\frac{D_{\nu^{(d_0)}}^{(d_0,2j\!-\!1)}}{D_{\nu^{(d_0)}}^{(d_0,2j\!-\!1)+}},   \,\nu^{(d_0)}_j=\mu^{(d_0)}_j;\\
       \frac{1}{{ \mathfrak{s}}^{(d_0,2j)}_{\mu^{(d_0)}\!-\!1}{ D}^{(d_0,2j)}_{\mu^{(d_0)}}}
    \begin{vmatrix}
   \! { \mathfrak{s}}^{(d_0,2j)}_{0} \!&\! \ldots \!& \!{ \mathfrak{s}}^{(d_0,2j)}_{\mu^{(d_0)}\!-\!1} \!&\! { \mathfrak{s}}^{(d_0,2j)}_{\mu^{(d_0)}} \!\\
    \!\cdots \!& \!\cdots \!&\! \cdots \!&\! \cdots \!\\
   \! { \mathfrak{s}}^{(d_0,2j)}_{\mu^{(d_0)}\!-\!1} \!& \!\ldots \!&\! { \mathfrak{s}}^{(d_0,2j)}_{2\mu^{(d_0)}\!-\!2} \!& \!{ \mathfrak{s}}^{(d_0,2j)}_{2\mu^{(d_0)}\!-\!1} \!\\
    1 \!&\! \ldots \!&\! \!\prod\limits_{j=1}^n\!\!z_j^{\mu^{(d_0)}-1} \!\!& \!\prod\limits_{j=1}^n\!\!z_j^{\mu^{(d_0)}}\! \!\\
    \end{vmatrix}, \,\nu^{(d_0)}_j\!<\!\mu^{(d_0)}_j,\\
    \end{array}
    \right.
\end{equation}
where the recurrence sequences $\left\{\mathfrak{s}^{(d_0,2j)}_{i}\right\}_{i=-1}^{2\left(\nu^{(d_0)}_N-\nu^{(d_0)}_j\right)-2}$ and
$\left\{\mathfrak s^{(d_0,2j-1)}_{i}\right\}_{i=0}^{2\left(\nu^{(d_0)}_N-\mu^{(d_0)}_j\right)-2}$ can be  found by  
\begin{equation}\label{20.3.11}\begin{split}
   &I_{\ell^{(d_0)}_j-\nu^{(d_0)}_j+2}=\\&= T(m_{\nu^{(d_0)}_j\!-\!1}^{(d_0,j)},\ldots,m_0^{(d_0,j)},\!-\!{\mathfrak s}^{(d_0,2j)}_{\!-\!1},\ldots,-
 {\mathfrak s}^{(d_0,2j)}_{\ell^{(d_0)}_j\!-\!2\nu^{(d_0)}_j}) \, T(\mathfrak s^{(d_0,2j\!-\!1)}_{\nu^{(d_0)}_j\!-\!1},\dots,\mathfrak s^{(d_0,2j\!-\!1)}_{\ell^{(d_0)}_j}),\end{split}
\end{equation}
where $ \ell_{j}^{(d_0)}=2\nu_N^{(d_0)}-2\mu^{(d_0)}_{j-1}-2$  and $\mathfrak s_i^{(d_0, 1)}=\mathfrak s_i^{(d_0)}$.
\begin{equation}\label{20.3.12}\begin{split}&
     T(l_{\mu^{(d_0)}_j-\nu^{(d_0)}_j}^{(d_0,j)},\ldots,l_0^{(d_0j)},-{\mathfrak s}^{(d_0,2j+1)}_{0},\ldots,-
 {\mathfrak s}^{(d_0,2j+1)}_{2\nu_N^{(d_0)}-2\mu^{(d_0)}_j-2}) \times\\&\times T({\mathfrak s}^{(d_0,2j)}_{\mu^{(d_0)}_j-\nu^{(d_0)}_j-1},
 \dots,{\mathfrak {s}}^{(d_0,2j)}_{\ell^{(d_0)}_j\!-\!2\nu^{(d_0)}_j})=I_{2\nu^{(d_0)}_N-\mu^{(d_0)}_j-\nu^{(d_0)}_j}.\end{split}
\end{equation}
\end {proposition}
\begin{proof}  Let the sequence of real numbers $\mathbf{s}=\{s_{i_1,i_2,...,i_n}\}_{i_1,...,i_n=0}^{2\nu_N^{(d_0)}-2}$ satisfy \eqref{20.3.1}. Consequently, the associated function $F_{d_0}$ takes the  form
\[\begin{split}
	F_{d_0}(z_1,...,z_n)&\!=\!-\cfrac{s_{0,...,0}}{\prod\limits_{i=1}^n z_i}\!-\!\cfrac{\binom{n}{1,...,1}s_{1,...,1}}{\prod\limits_{i=1}^n z_i^{2}}-\ldots-
	\cfrac{\binom{n\times (2\nu^{(d_0)}_N-2)}{2\nu^{(d_0)}_N-2,...,2\nu^{(d_0)}_N-2}s_{2\nu^{(d_0)}_N-2,...,2\nu^{(d_0)}_N-2}}{\prod\limits_{i=1}^n z_i^{2\nu^{(d_0)}_N-1}}+\\&+o\left(\cfrac{1}{\prod\limits_{i=1}^n z_i^{2\nu^{(d_0)}_N-1}}\right).
\end{split}\]
Let us define   $d_0$-associated sequence $\mathfrak s^{(d_0)}=\left\{\mathfrak s^{(d_0)}_i\right\}_{i=0}^{2\nu_N^{(d_0)}-2}$ by \eqref{20.3.3}  and let $\mathcal{N}\left({\mathfrak s^{(d_0)}}\right)=\left\{\!\nu_j^{(d_0)}\!\right\}_{j\!=\!1}^N\!\!\cup\!\left\{\!\mu_j^{(d_0)}\!\right\}_{j=1}^{N\!-\!1}$ be the set of normal indices of  $\mathfrak s^{(d_0)}$. Then $F_{d_0}$ can be rewritten as
 \begin{equation}\label{20.3.14ddsz1}
	F_{d_0}(z)=-\cfrac{\mathfrak s^{(d_0)}_0}{z}-\cfrac{\mathfrak s^{(d_0)}_1}{z^2}-\ldots-\cfrac{\mathfrak s^{(d_0)}_{2\nu^{(d_0)}_N-2}}{z^{2\nu^{(d_0)}_N-1}}+o\left(
	\cfrac{1}{z^{2\nu^{(d_0)}_N-1}}\right), \quad\mbox{where}\quad z=\prod\limits_{i=1}^n z_i.
\end{equation}
By Theorem \ref{2p.pr.alg1}, we obtain \eqref{20.3.8}--\eqref{20.3.12}.  This completes the proof.~\end{proof}
\begin{remark}\label{20p. cor 3.3} If  $d_0$-associated sequence $\mathfrak s^{(d_0)}=\{\mathfrak s^{(d_0)}_i\}_{i=0}^{2\nu_N^{(d_0)}-2}$  is  regular (see Definition \ref{def:3p.4.3}), then $l_j^{(d_0)}=$const and
 \begin{equation}\label{20.3.14x1}
	l_j^{(d_0)}
   =   \frac{1}{{\mathfrak{s}}_{-1}^{(d_0,2j)}}\qquad\mbox{for all } j=\overline{1,N-1}.
   \end{equation}
\end{remark}

Next we define Stieltjes polynomials of the first and second kind, which will be used to describe the solution.
\begin{definition}\label{20p. def3.4} (\cite{K17}) The polynomials $P_{j,d_0}^{+}$ and $Q_{j,d_0}^{+}$ can be defined similar as the solutions of the following system
 \begin{equation}\label{20.3.14xz1}
   \left\{
    \begin{array}{rcl}
    y_{2j,d_0}-y_{2j-2,d_0}=l^{(d_0)}_{j}(z_1,...,z_n)y_{2j-1,d_0},\\
    y_{2j+1,d_0}-y_{2j-1,d_0}=-m^{(d_0)}_{j+1}(z_1,...,z_n)y_{2j,d_0}\prod\limits_{k=1}^nz_k\\
    \end{array}\right.
   \end{equation}
subject to the initial conditions 
 \begin{equation}\label{20.3.14xz2}\begin{split}&
    P_{-1,d_0}^{+}(z_1,...,z_n)\equiv0\quad\mbox{and
    }\quad P_{0,d_0}^{+}(z_1,...,z_n)\equiv1,\\&
    Q_{-1,d_0}^{+}(z_1,...,z_n)\equiv 1\quad\mbox{and
    }\quad Q_{0,d_0}^{+}(z_1,...,z_n)\equiv0.
    \end{split}
   \end{equation}
   
    $P_{j,d_0}^{+}$ and $Q_{j,d_0}^{+}$ are called  $d_0$--Stieltjes polynomials of the first and second kinds, respectively.\end{definition}

 \begin{proposition}\label{20p.prop3.7} Let $\mathbf{s}=\{s_{i_1,i_2,...,i_n}\}_{i_1,...,i_n=0}^{2\nu_N^{(d_0)}-2}$ be a sequence of real numbers such that  \eqref{20.3.1} holds. Let  $\mathfrak s^{(d_0)}=\{\mathfrak s^{(d_0)}_i\}_{i=0}^{2\nu_N^{(d_0)}-2}$ be  $d_0$-associated sequence of $\mathbf{s}$ and let $\mathcal{N}({\mathfrak s^{(d_0)}})=\{\nu_j^{(d_0)}\}_{j=1}^N\cup\{\mu_j^{(d_0)}\}_{j=1}^{N-1}$ be the set of normal indices of $\mathfrak s^{(d_0)}$. Then:
  \begin{enumerate}
 \item Any solution of   $d_0-$diagonal truncated problem $\bold{MP}(\mathbf{s},2\nu^{(d_0)}_N-2)$ admits the following representation in terms of  $d_0-$Stieltjes polynomials by
 \begin{equation}\label{20.3.prp3.5}
F_{d_0}(z_1,...,z_n)=\cfrac{Q^+_{2N-1,d_0}(z_1,...,z_n)\tau(z_1,...,z_n)+Q^+_{2N-2,d_0}(z_1,...,z_n)}{P^+_{2N-1,d_0}(z_1,...,z_n)\tau(z_1,...,z_n)+P^+_{2N-2,d_0}(z_1,...,z_n)},
\end{equation}
 where the parameter $\tau$ satisfies \eqref{20.3.9}.

 \item The resolvent matrix $W_{2N-1,d_0}$ of  $d_0-$diagonal truncated problem $\bold{MP}(\mathbf{s},2\nu^{(d_0)}_N~-~2)$ can be represented in terms of   $d_0-$Stieltjes polynomials by
  \begin{equation}\label{20.3.prp3.5.1}
  	W_{2N-1,d_0}(z_1,...,z_n)=\begin{pmatrix}
Q^+_{2N-1,d_0}(z_1,...,z_n) & Q^+_{2N-2,d_0}(z_1,...,z_n)  \\
P^+_{2N-1,d_0}(z_1,...,z_n) & P^+_{2N-2,d_0}(z_1,...,z_n)
\end{pmatrix}.
  \end{equation}
  
  Furthermore, the matrix valued function  $W_{2N-1,d_0}$ admits the following factorization
    \begin{equation}\label{20.3.prp3.5.2}\begin{split}
  	W_{2N-1,d_0}(z_1,...,z_n)&=M_{1,d_0}(z_1,...,z_n)\times L_{1,d_0}(z_1,...,z_n)\times\ldots\times\\& \times L_{N-1,d_0}(z_1,...,z_n)\times M_{N,d_0}(z_1,...,z_n),
  \end{split}  \end{equation}
  where the matrices $M_{j,d_0}$ and $L_{j,d_0}$ are defined by
    \begin{equation}\label{20.3.prp3.5.3}\begin{split}&
  M_{j,d_0}(z_1,...,z_n)=\begin{pmatrix}
1& 0\\
-m_j^{(d_0)}(z_1,...,z_n)\!\!\prod\limits_{i=1}^n\!\! z_i& 1
\end{pmatrix}, \\&L_{j,d_0}(z_1,...,z_n)=\begin{pmatrix}
1& l_j^{(d_0)}(z_1,...,z_n)\\
0& 1
\end{pmatrix}.\end{split}
  \end{equation}
 \end{enumerate}
 \end{proposition}
\begin{proof} Suppose the sequence of real numbers $\mathbf{s}=\{s_{i_1,i_2,...,i_n}\}_{i_1,...,i_n=0}^{2\nu_N^{(d_0)}-2}$ satisfies  \eqref{20.3.1}  and $s^{(d_0)}=\{\mathfrak s^{(d_0)}_i\}_{i=0}^{2\nu_N^{(d_0)}-2}$ is  $d_0$-associated sequence of $\mathbf{s}$, $\mathcal{N}({\mathfrak s^{(d_0)}})=\{\nu_j^{(d_0)}\}_{j=1}^N\cup\{\mu_j^{(d_0)}\}_{j=1}^{N-1}$ is the  set of normal indices of   $\mathfrak s^{(d_0)}$. Then $F_{d_0}$ admits the asymptotic expansion \eqref{20.3.14ddsz1}. By \cite[Theorem 5.2]{K17},  we obtain \eqref{20.3.prp3.5}-\eqref{20.3.prp3.5.3}. This completes the proof.~\end{proof}

Next, we study the even case:
\begin{proposition}\label{20.th3.4} Let $\mathbf{s}=\{s_{i_1,i_2,...,i_n}\}_{i_1,...,i_n=0}^{2\mu_N^{(d_0)}-1}$ be a sequence of real numbers such that  \eqref{20.3.1} holds. Let  $\mathfrak s^{(d_0)}=\{\mathfrak s^{(d_0)}_i\}_{i=0}^{2\mu_N^{(d_0)}-1}$ be  $d_0$-associated sequence of $\mathbf{s}$ and let $\mathcal{N}({\mathfrak s^{(d_0)}})=\{\nu_j^{(d_0)}\}_{j=1}^N\cup\{\mu_j^{(d_0)}\}_{j=1}^{N}$ be the set of normal indices  of $\mathfrak s^{(d_0)}$. Then any solution of $d_0$--diagonal truncated problem $\bold{MP}(\mathbf{s},2\mu^{(d_0)}_N-1)$ admits the representation
 \begin{equation}\label{20.3.14}\begin{split}
 	F_{d_0}(z_1,..., z_n)&=\cfrac{1|}{\big|-m_1^{(d_0)}(z_1,..., z_n)\prod\limits_{j=1}^nz_j}+\cfrac{1|}{\big|l_1^{(d_0)}(z_1,...,z_n)}+\ldots+\\&+
	\cfrac{1|}{\big|-m_N^{(d_0)}(z_1,..., z_n)\prod\limits_{j=1}^nz_j}+\cfrac{1|}{\big|l_N^{(d_0)}(z_1,...,z_n)+\tau(z_1,...,z_n)},
 \end{split}\end{equation}
 where  the parameter $\tau$ satisfy the following  
\begin{equation}\label{20.3.15}
  	\tau(z_1, ..., z_n)=o\left(1\right),
   \end{equation}
 the atoms $(m_j^{(d_0)} , l^{(d_0)}_j)$ can be found by~\eqref{eq:}--\eqref{20.3.10}.
\end{proposition}
\begin{proof}  Let  $\mathbf{s}=\{s_{i_1,i_2,...,i_n}\}_{i_1,...,i_n=0}^{2\mu_N^{(d_0)}-1}$ satisfy \eqref{20.3.1}. Then
\[\begin{split}
	F_{d_0}(z_1,...,z_n)=&-\cfrac{s_{0,...,0}}{\prod\limits_{i=1}^n z_i}-\cfrac{\binom{n}{1,...,1}s_{1,...,1}}{\prod\limits_{i=1}^n z_i^{2}}-\ldots-\\&-
	\cfrac{\binom{n\times (2\mu^{(d_0)}_N-1)}{2\mu^{(d_0)}_N-1,...,2\mu^{(d_0)}_N-1}s_{2\mu^{(d_0)}_N-1,...,2\mu^{(d_0)}_N-1}}{\prod\limits_{i=1}^n z_i^{2\mu^{(d_0)}_N}}+o\left(\cfrac{1}{\prod\limits_{i=1}^n z_i^{2\mu^{(d_0)}_N}}\right).
\end{split}\]
Let  $d_0$-associated sequence $\mathfrak s^{(d_0)}=\{\mathfrak s^{(d_0)}_i\}_{i=0}^{2\mu_N^{(d_0)}-1}$ be defined by \eqref{20.3.3}  and $\mathcal{N}({\mathfrak s^{(d_0)}})=\{\nu_j^{(d_0)}\}_{j=1}^N\cup\{\mu_j^{(d_0)}\}_{j=1}^{N}$ be  the set of normal indices of  the sequence $\mathfrak s^{(d_0)}$.  Therefore, $F_{d_0}$ can be rewritten as
\begin{equation}\label{20.3.qqw_1}
	F(z)=-\cfrac{\mathfrak s^{(d_0)}_0}{z}-\cfrac{\mathfrak s^{(d_0)}_1}{z^2}-\ldots-\cfrac{\mathfrak s^{(d_0)}_{2\mu^{(d_0)}_N-1}}{z^{2\mu^{(d_0)}_N}}+o\left(
	\cfrac{1}{z^{2\mu^{(d_0)}_N}}\right), \quad\mbox{where}\quad z=\prod\limits_{i=1}^n z_i.
   \end{equation}
By Theorem \ref{14p.int.th_ind**}, we obtain \eqref{20.3.14}--\eqref{20.3.15}.  This completes the proof.~\end{proof}
\begin{remark}\label{20p. cor 3.5} If  $d_0$-associated sequence $\mathfrak s^{(d_0)}=\{\mathfrak s^{(d_0)}_i\}_{i=0}^{2\mu_N^{(d_0)}-1}$  is the regular sequence, then  all $l_j^{(d_0)}=$const and
 \begin{equation}\label{20.3.14y1}
	l_j^{(d_0)}
   =   \frac{1}{{\mathfrak{s}}_{-1}^{(d_0,2j)}}\qquad\mbox{for all } j=\overline{1,N}.
   \end{equation}
\end{remark}

 \begin{proposition}\label{20p.prop3.8}Let $\mathbf{s}=\{s_{i_1,i_2,...,i_n}\}_{i_1,...,i_n=0}^{2\mu_N^{(d_0)}-1}$ be the sequence of real numbers such that  \eqref{20.3.1} holds. Let  $\mathfrak s^{d_0}=\{\mathfrak s^{(d_0)}_i\}_{i=0}^{2\mu_N^{(d_0)}-1}$ be $d_0$-associated sequence of $\mathbf{s}$ and let $\mathcal{N}({\mathfrak s^{(d_0)}})=\{\nu_j^{(d_0)}\}_{j=1}^N\cup\{\mu_j^{(d_0)}\}_{j=1}^{N}$ be the  set of normal indices of  $\mathfrak s^{(d_0)}$. Then:
   \begin{enumerate}
 \item Any solution of $d_0$--diagonal truncated problem $\bold{MP}(\mathbf{s},2\mu^{(d_0)}_N-1)$ admits the representation in terms of  $d_0-$Stieltjes polynomials
 \begin{equation}\label{20.3.prp3.8} by
F_{d_0}(z_1,...,z_n)=\cfrac{Q^+_{2N-1,d_0}(z_1,...,z_n)\tau(z_1,...,z_n)+Q^+_{2N,d_0}(z_1,...,z_n)}{P^+_{2N-1,d_0}(z_1,...,z_n)\tau(z_1,...,z_n)+P^+_{2N,d_0}(z_1,...,z_n)},
\end{equation}
 where the parameter $\tau$ satisfies \eqref{20.3.15}.
 \item the resolvent matrix $W_{2N,d_0}$ of  $d_0-$diagonal truncated problem $\bold{MP}(\mathbf{s},2\mu^{(d_0)}_N~-~1)$ can be represented in terms of   $d_0-$Stieltjes polynomials by
  \begin{equation}\label{20.3.prp3.8.1}
  	W_{2N,d_0}(z_1,...,z_n)=\begin{pmatrix}
Q^+_{2N-1,d_0}(z_1,...,z_n) & Q^+_{2N,d_0}(z_1,...,z_n)  \\
P^+_{2N-1,d_0}(z_1,...,z_n) & P^+_{2N,d_0}(z_1,...,z_n)
\end{pmatrix}.
  \end{equation}
  
  Furthermore,   $W_{2N,d_0}$ admits the following factorization
    \begin{equation}\label{20.3.prp3.8.2}\begin{split}
  	W_{2N,d_0}(z_1,...,z_n)=&M_{1,d_0}(z_1,...,z_n)\times L_{1,d_0}(z_1,...,z_n)\times\ldots\times \\&\times M_{N,d_0}(z_1,...,z_n)\times L_{N,d_0}(z_1,...,z_n),
   \end{split} \end{equation}
  where the matrices $M_{j,d_0}$ and $L_{j,d_0}$ are defined by \eqref{20.3.prp3.5.3}.
 \end{enumerate}
 \end{proposition}
\begin{proof} Assume $\mathbf{s}=\{s_{i_1,i_2,...,i_n}\}_{i_1,...,i_n=0}^{2\mu_N^{(d_0)}-1}$ satisfies \eqref{20.3.1}, $s^{d_0}=\{\mathfrak s^{(d_0)}_i\}_{i=0}^{2\mu_N^{(d_0)}-1}$ is  $d_0$-associated sequence of $\mathbf{s}$ and  $\mathcal{N}({\mathfrak s^{(d_0)}})=\{\nu_j^{(d_0)}\}_{j=1}^N\cup\{\mu_j^{(d_0)}\}_{j=1}^{N}$ is the  set  of normal indices of  $\mathfrak s^{(d_0)}$. Consequently, $F_{d_0}$ admits \eqref{20.3.qqw_1}. By \cite[Theorem 5.5]{K17}, we obtain \eqref{20.3.prp3.8}--\eqref{20.3.prp3.8.2}.  This completes the proof.~\end{proof}

\subsection{$d_0$-diagonal full problem} Let us formulate  $d_0$-diagonal full problem $\bold{MP}(\mathbf{s})$:
 
  Given  the sequence of real numbers  $\mathbf{s}=\{s_{i_1,i_2,...,i_n}\}_{i_1,...,i_n=0}^{\infty}$, such that \eqref{20.3.1} holds. Describe the set of functions $F_{d_0}$, which admits the following asymptotic expansion
\begin{equation}\label{20.3.eert_3}
	F_{d_0}(z_1,...,z_n)=-\cfrac{s_{0,...,0}}{\prod\limits_{i=1}^n z_i}-\cfrac{\binom{n}{1,...,1}s_{1,...,1}}{\prod\limits_{i=1}^n z_i^{2}}-
	\cfrac{\binom{2n}{2,...,2}s_{2,...,2}}{\prod\limits_{i=1}^n z_i^{3}}-\ldots.
\end{equation}

\begin{proposition}\label{20pxprop3.2} Let the sequence of real numbers  $\mathbf{s}=\{s_{i_1,i_2,...,i_n}\}_{i_1,...,i_n=0}^{\infty}$ satisfy \eqref{20.3.1}. Let  $\mathfrak s^{(d_0)}=\left\{\mathfrak s^{(d_0)}_i\right\}_{i=0}^{\infty}$ be  $d_0$-associated sequence of $\mathbf{s}$ and   regular sequence, let $\mathcal{N}\left({\mathfrak s^{(d_0)}}\right)$ be the set of normal indices  of $\mathfrak s^{(d_0)}$. Assume in addition $l_j^{(d_0)}>0$ for all $j\in \mathbb{N}$.
Then:
   \begin{enumerate}
 \item Any solution of   $d_0$-diagonal full problem $\bold{MP}(\mathbf{s})$ admits the representation
\begin{equation}\label{20.3.eert_1}
	F_{d_0}(z_1,...z_n)=\underset{i=0}{\overset{\infty}{\mathbf{K}}}\left(-\cfrac{b_{i}^{(d_0)}}{a_{i}^{(d_0)}(z_1,...,z_n)}\right),
	\end{equation}
where the atoms $(a_i^{(d_0)}, b_i^{(d_0)})$ are defined by~\eqref{20.3.7}.

 \item  $d_0$-diagonal  full problem  $\bold{MP}(\mathbf{s})$  is the indeterminate if and only if 
 \begin{equation}\label{20.3.eert_2}
 	\sum_{i=0}^{\infty}\!\!\cfrac{\left|P_{i}^{(d_0)}(0,...,0)\right|^2}{\tilde{b}_i^{(d_0)}}<\infty\quad\mbox{and}\quad 
	\sum_{i=0}^{\infty}\!\!\cfrac{\left|Q_{i}^{(d_0)}(0,...,0)\right|^2}{\tilde{b}_i^{(d_0)}}<\infty, \, \tilde{b}_i^{(d_0)}=b_0^{(d_0)}\cdot...\cdot  b_i^{(d_0)}.
 	\end{equation}
\end{enumerate}

\end{proposition}
\begin{proof} Let the sequence of real numbers  $\mathbf{s}=\{s_{i_1,i_2,...,i_n}\}_{i_1,...,i_n=0}^{\infty}$ satisfy \eqref{20.3.1} and let    $d_0$-associated sequence $\mathfrak s^{(d_0)}=\left\{\mathfrak s^{(d_0)}_i\right\}_{i=0}^{\infty}$   be  regular. Consequently $F_{d_0}$ admits the following asymptotic expansion
\begin{equation}\label{20.3.eert_4}
	F_{d_0}(z_1,...,z_n)=-\cfrac{s_{0,...,0}}{\prod\limits_{i=1}^n z_i}-\cfrac{\binom{n}{1,...,1}s_{1,...,1}}{\prod\limits_{i=1}^n z_i^{2}}-
	\cfrac{\binom{2n}{2,...,2}s_{2,...,2}}{\prod\limits_{i=1}^n z_i^{3}}-\ldots=-\cfrac{\mathfrak s^{(d_0)}_0}{z}-\cfrac{\mathfrak s^{(d_0)}_1}{z^2}-\ldots.
\end{equation}

According to Theorem \ref{20.th3.1} , we apply the infinite numbers of step to \eqref{20.3.eert_4} and obtain \eqref{20.3.eert_1}. By \cite[Proposition 4.3]{DK20}, we get that   $d_0$-diagonal full problem  $\bold{MP}(\mathbf{s})$  is the indeterminate iff \eqref{20.3.eert_2} holds. 
This completes the proof.~\end{proof}
\begin{theorem}\label{20.th3.6} Let the sequence of real numbers  $\mathbf{s}=\{s_{i_1,i_2,...,i_n}\}_{i_1,...,i_n=0}^{\infty}$ satisfy \eqref{20.3.1}. Let  $\mathfrak s^{(d_0)}=\left\{\mathfrak s^{(d_0)}_i\right\}_{i=0}^{\infty}$ be  $d_0$-associated sequence of $\mathbf{s}$ and   regular, let $\mathcal{N}\left({\mathfrak s^{(d_0)}}\right)$ be the set of normal indices  of $\mathfrak s^{(d_0)}$. Then:
   \begin{enumerate}
 \item Any solution of  $d_0$-diagonal full problem $\bold{MP}(\mathbf{s})$ admits the representation
\begin{equation}\label{20.3.19}
	F_{d_0}(z_1,...z_n)=\underset{i=1}{\overset{\infty}{\mathbf{K}}}\left(\cfrac{1}{-m_i^{(d_0)}(z_1,..., z_n)\prod\limits_{j=1}^nz_j+\cfrac{1}{l_i^{(d_0)}}}\right),
	\end{equation}
where the atoms $(m_i^{(d_0)}, l_i^{(d_0)})$ are defined by~\eqref{eq:}--\eqref{20.3.10}.

 \item   $d_0$-diagonal full problem  $\bold{MP}(\mathbf{s})$  is the indeterminate if and only if 
\begin{equation}\label{20.3.20}
	\sum_{i=1}^\infty m_i^{(d_0)}(0,...,0)<\infty\quad\mbox{and}\quad \sum_{i=1}^\infty l_i^{(d_0)}<\infty.
\end{equation}

\end{enumerate}

Furthermore, if \eqref{20.3.20} holds, then the sequence of resolvents matrices $W_{2n,d_0}$ converges to the  matrix valued function $W^{+}_{\infty, d_0}(z_1,...,z_n)=\left(w^+_{ij,d_0}(z_1,...,z_n)\right)_{i,j=1}^2$ and $F_{d_0}$  can be represented as follows
\begin{equation}\label{20.3.19x_1}
	F_{d_0}(z_1,...z_n)=\cfrac{w^+_{11,d_0}(z_1,...,z_n)\tau(z_1,...,z_n)+w^+_{12,d_0}(z_1,...,z_n)}{w^+_{21,d_0}(z_1,...,z_n)\tau(z_1,...,z_n)+w^+_{22,d_0}(z_1,...,z_n)},
	\end{equation}
where the parameter $\tau(z_1, ..., z_n)=o\left(1\right)$.
\end{theorem}

\begin{proof} Let  $\mathbf{s}=\{s_{i_1,i_2,...,i_n}\}_{i_1,...,i_n=0}^{\infty}$ satisfy \eqref{20.3.1} and let $\mathfrak s^{(d_0)}=\left\{\mathfrak s^{(d_0)}_i\right\}_{i=0}^{\infty}$  be its $d_0$-associated sequence, such that $\mathfrak s^{(d_0)}$ is the regular sequence. Consequently  any solution of $d_0$-diagonal full problem $\bold{MP}(\mathbf{s})$ has got the asymptotic expansion \eqref{20.3.eert_3}

By Proposition \ref{20.th3.3}  ( or Proposition \ref{20.th3.4})  with the  infinite number of steps, we obtain representation \eqref{20.3.19}, where the the atoms $(m_i^{(d_0)}, l_i^{(d_0)})$  can be calculated by similar formulas ~\eqref{eq:}--\eqref{20.3.10}. Due to \cite[Theorem 5.3]{DK17}, $d_0$-diagonal full problem $\bold{MP}(\mathbf{s})$  is indeterminate if and only if  \eqref{20.3.20} holds.  After that, by \cite[Theorem 4.2]{DK20}, we obtain the convergence of the sequence $W_{2n,d_0}$  to   the  matrix valued function $W^{+}_{\infty, d_0}(z_1,...,z_n)=\left(w^+_{ij,d_0}(z_1,...,z_n)\right)_{i,j=1}^2$ and representation  \eqref{20.3.19x_1}.
This completes the proof.~\end{proof}

\section{$d_{j_1,...,j_n}$-diagonal types}

In the present section,  we study the case, when  all moments not belonging to a fixed subdiagonal  of the Hankel matrix are equal to zero. In the current case,  we get a situation similar to the previous section. Of course, one is the basic problem for the general problem.

First of all, we define a subdiagonal sequence and formulate associated problem.

\begin{definition}\label{20p.def4.1}Let $\mathbf{s}=\{s_{i_1,i_2,...,i_n}\}_{i_1,...,i_n=0}^{\ell}$ be a sequence of real numbers and let  us set
\begin{equation}\label{20p.4.1}
\mathfrak{s}^{(d_{j_1,...,j_n})}_{j}=\binom{j\times n+\sum\limits_{k=1}^nj_k}{j_1+j,...,j_n+j}s_{j_1+j,...,j_n+j},
\end{equation}
then $\mathfrak{s}^{(d_{j_1,...,j_n})}=\left\{\mathfrak{s}^{(d_{j_1,...,j_n})}_{i}\right\}_{i=0}^{\ell}$   is called  $d_{j_1,...,j_n}$ -- associated sequence of $\mathbf{s}$ if  there is at least one $j_k=0$.

It is the special case of the sequence   $\mathbf{s}=\{s_{i_1,i_2,...,i_n}\}_{i_1,...,i_n=0}^{\ell}$, where  the moments $s_{i_1,i_2,...,i_n}$ satisfy the following 
\begin{equation}\label{20p.4.1Xqw}
	\mbox{if for at least one index} \, i_k\neq j_j+j \,\mbox{holds, then}\,\,s_{i_1,i_2,...,i_n}=0,
\end{equation}
i.e. if the moment $s_{i_1,i_2,...,i_n}$ does not belong to $d_{j_1,...,j_n}$ -- associated sequence of $\mathbf{s}$, then $s_{i_1,i_2,...,i_n}=0$. 

In this case, the associated function is  denoted by $F_{d_{j_1,...,j_n}}$,. i.e. the function  $F_{d_{j_1,...,j_n}}$  is the associated with the sequence  $\mathfrak{s}^{(d_{j_1,...,j_n})}=\left\{\mathfrak{s}^{(d_{j_1,...,j_n})}_{i}\right\}_{i=0}^{\ell}$ and  $F_{d_{j_1,...,j_n}}$ takes the following asymptotic expansion 
\begin{equation}\label{20p.4.2}
	F_{d_{j_1,...,j_n}}(z_1,...,z_n)\!=\!-\cfrac{\mathfrak{s}^{(d_{j_1,...,j_n})}_{0}}{\prod\limits_{i=1}^n z_i^{j_i+1}}\!-\!\cfrac{\mathfrak{s}^{(d_{j_1,...,j_n})}_{1}}{\prod\limits_{i=1}^n z_i^{j_i+2}}\!-\!\ldots
	\!-\!\cfrac{\mathfrak{s}^{(d_{j_1,...,j_n})}_{\ell}}{\prod\limits_{i=1}^n z_i^{j_i+\ell+1}}+o\!\left(\cfrac{1}{\prod\limits_{i=1}^n z_i^{j_i+\ell+1}
}\right).
\end{equation}
\end{definition}
 On the other hand, by \eqref{20p.4.1}, representation  \eqref{20p.4.2}  can be  rewritten as follows 
\begin{equation}\label{20p.4.3}
	F_{d_{j_1,...,j_n}}(z_1,...,z_n)=\cfrac{1}{\prod\limits_{i=1}^n z_i^{j_i}}\left(
	-\cfrac{\mathfrak{s}^{(d_{j_1,...,j_n})}_{0}}{\prod\limits_{i=1}^n z_i}-\ldots
	-\cfrac{\mathfrak{s}^{(d_{j_1,...,j_n})}_{\ell}}{\prod\limits_{i=1}^n z_i^{\ell+1}}+o\left(\cfrac{1}{\prod\limits_{i=1}^n z_i^{\ell+1}
}\right)
	\right).
\end{equation}

\textbf{$d_{j_1,...,j_n}$-diagonal truncated   problem $\bold{MP}(\mathbf{s},\ell)$ can be formulated as:}

Given the sequence of real numbers  $\mathbf{s}=\{s_{i_1,i_2,...,i_n}\}_{i_1,...,i_n=0}^\ell$ , which satisfies \eqref{20p.4.1Xqw} and $j_1,...,j_n \in \mathbb{Z}_{+}$ are fixed, such that at least one of them equals to zero (i.e. there exist some $j_k=0$) and  $\mathfrak{s}^{(d_{j_1,...,j_n})}=\left\{\mathfrak{s}^{(d_{j_1,...,j_n})}_{i}\right\}_{i=0}^{2n^{(d_{j_1,...,j_n})}_N-1}$ is  $d_{j_1,...,j_n}$-associated sequence of $\mathbf{s}$.  Describe the set of functions $F_{d_{j_1,...,j_n}}$, which  have the asymptotic expansion \eqref{20p.4.2}.

\begin{remark}By representation \eqref{20p.4.3}, we will obtain the similar results to $F_{d_0}$.\end{remark}
 \subsection{Odd $d_{j_1,...,j_n}$-diagonal truncated  problem}


\begin{proposition}\label{20.th4.3} Let $\mathbf{s}=\{s_{i_1,i_2,...,i_n}\}_{i_1,...,i_n=0}^{2\nu_N^{(d_{j_1,...,j_n})}-2}$ be the sequence of real numbers   and let $j_1,...,j_n \in \mathbb{Z}_{+}$ be fixed such that \eqref{20p.4.1Xqw} holds and  $\mathfrak{s}^{(d_{j_1,...,j_n})}=\left\{\mathfrak{s}^{(d_{j_1,...,j_n})}_{i}\right\}_{i=0}^{2n^{(d_{j_1,...,j_n})}_N-1}$ is  $d_{j_1,...,j_n}$-associated sequence of $\mathbf{s}$. Let   $\mathcal{N}\left({\mathfrak s^{(d_{j_1,...,j_n})}}\right)=\left\{\nu_j^{(d_{j_1,...,j_n})}\right\}_{j=1}^N\cup\left\{\mu_j^{(d_{j_1,...,j_n})}\right\}_{j=1}^{N-1}$ be the  set of normal indices of $\mathfrak s^{(d_{j_1,...,j_n})}$. Then any solution of  $d_{j_1,...,j_n}$-diagonal truncated problem $\bold{MP}\left(\mathbf{s},2\nu^{(d_{j_1,...,j_n})}_N-2\right)$ admits the representation
 \begin{equation}\label{20.4.7}\begin{split}
 	F_{d_{j_1,...,j_n}}\!(\!z_1,..., \!z_n\!)\!=\!&\cfrac{1}{\prod\limits_{i=1}^n z_i^{j_i}}\!\cdot\!\cfrac{1\big|}{\big|\!-\!m_1^{(d_{j_1,...,j_n})}\!(\!z_1,..., \!z_n\!)\!\!\prod\limits_{j=1}^n\!z_j}\!+\!\cfrac{1\big|}{\big|l_1^{(d_{j_1,...,j_n})}\!(\!z_1,...,\!z_n\!)}\!+\!\ldots\!+\!\\&+\cfrac{1\big|}{\big| -m_N^{(d_{j_1,...,j_n})}(z_1,..., z_n)\prod\limits_{j=1}^nz_j+\cfrac{1}{\tau(z_1,...,z_n)}},
	\end{split}
 \end{equation}
 where  the parameter $\tau$ satisfy the following condition
  \begin{equation}\label{20.4.8}
  \cfrac{1}{\tau(z_1, ..., z_n)}=o\left(\prod\limits_{j=1}^nz_j\right),
   \end{equation}
  the atoms $(m_j^{(d_{j_1,...,j_n})} , l^{(d_{j_1,...,j_n})}_j)$ can be calculated by
 \begin{equation}\label{eq4.9:}\begin{split}
       & m_j^{(d_{j_1,...,j_n})}(z_1,..., z_n)=\frac{(-1)^{\nu^{(d_{j_1,...,j_n})}+1}}{D^{(d_{j_1,...,j_n},2j-1)}_{\nu^{(d_{j_1,...,j_n})}}}\times\\&
       \times \begin{vmatrix}
            0 & \ldots & 0 &  \mathfrak s^{(d_0,2j-1)}_{\nu^{(d_{j_1,...,j_n})}-1} &  \mathfrak s^{(d_{j_1,...,j_n},2j-1)}_{\nu^{(d_{j_1,...,j_n})}} \\
            \vdots &  & \ldots & \ldots & \vdots \\
            \mathfrak s^{(d_{j_1,...,j_n},2j-1)}_{\nu^{(d_{j_1,...,j_n})}-1} &\mathfrak s^{(d_{j_1,...,j_n},2j-1)}_{\nu^{(d_{j_1,...,j_n})}} & \ldots & \ldots & \mathfrak s^{(d_{j_1,...,j_n},2j-1)}_{2\nu^{(d_{j_1,...,j_n})}-2} \\
            1 & \prod\limits_{j=1}^nz_j & \ldots& \prod\limits_{j=1}^nz_j^{\nu^{(d_{j_1,...,j_n})}-2}& \prod\limits_{j=1}^nz_j^{\nu^{(d_{j_1,...,j_n})}-1} \\
        \end{vmatrix},
\end{split}\end{equation}

\begin{equation}\label{20.4.10}\begin{split}
 &\mbox{if}\quad  \nu^{(d_{j_1,...,j_n})}_j=\mu^{(d_{j_1,...,j_n})}_j,\mbox{then}\\& l_j^{(d_{j_1,...,j_n})}(z_1,..., z_n)
   = (-1)^{\nu^{(d_{j_1,...,j_n})}
    +1} \mathfrak s_{\nu^{(d_{j_1,...,j_n})}-1}^{(d_{j_1,...,j_n},2j-1)}\frac{D_{\nu^{(d_{j_1,...,j_n})}}^{(d_{j_1,...,j_n},2j-1)}}{D_{\nu^{(d_{j_1,...,j_n})}}^{(d_{j_1,...,j_n},2j-1)+}};\\&
    \mbox{if}\quad  \nu^{(d_{j_1,...,j_n})}_j<\mu^{(d_{j_1,...,j_n})}_j,\mbox{then}\\& 
l_j^{(d_{j_1,...,j_n})}(z_1,..., z_n)
   = \! C\!
    \begin{vmatrix}
    { \mathfrak{s}}^{(d_{j_1,...,j_n},2j)}_{0} & \ldots & { \mathfrak{s}}^{(d_{j_1,...,j_n},2j)}_{\mu^{(d_{j_1,...,j_n})}-1} & { \mathfrak{s}}^{(d_{j_1,...,j_n},2j)}_{\mu^{(d_{j_1,...,j_n})}} \\
    \cdots & \cdots & \cdots & \cdots \\
    { \mathfrak{s}}^{(d_{j_1,...,j_n},2j)}_{\mu^{(d_{j_1,...,j_n})}-1} & \ldots & { \mathfrak{s}}^{(d_{j_1,...,j_n},2j)}_{2\mu^{(d_{j_1,...,j_n})}-2} & { \mathfrak{s}}^{(d_{j_1,...,j_n},2j)}_{2\mu^{(d_{j_1,...,j_n})}-1} \\
    1 & \ldots & \prod\limits_{j=1}^n\!z_j^{\mu^{(d_{j_1,...,j_n})}-1} & \prod\limits_{j=1}^n\!z_j^{\mu^{(d_{j_1,...,j_n})}} \\
    \end{vmatrix},\\
\end{split}\end{equation}
where the constant $C$ is defined by 
\[C=  \frac{1}{{ \mathfrak{s}}^{(d_{j_1,...,j_n},2j)}_{\mu^{(d_{j_1,...,j_n})}-1}{ D}^{(d_{j_1,...,j_n},2j)}_{\mu^{(d_{j_1,...,j_n})}}}\]
 and  the recurrence sequence $\mathfrak{s}^{(d_{j_1,...,j_n},2j)}=\left\{\mathfrak{s}^{(d_{j_1,...,j_n},2j)}_{i}\right\}_{i=-1}^{2\left(\nu^{(d_{j_1,...,j_n})}_N-\nu^{(d_{j_1,...,j_n})}_j\right)-2}$ and
$\mathfrak{s}^{(d_{j_1,...,j_n},2j-1)}=\left\{\mathfrak s^{(d_{j_1,...,j_n},2j-1)}_{i}\right\}_{i=0}^{2\left(\nu^{(d_{j_1,...,j_n})}_N-\mu^{(d_{j_1,...,j_n})}_j\right)-2}$ can be  found by  
\begin{equation}\label{20.4.11}\begin{split}&
    T(m_{\nu^{(d_{j_1,...,j_n})}_j-1}^{(d_{j_1,...,j_n},j)},\ldots,m_0^{(d_{j_1,...,j_n},j)},-{\mathfrak s}^{(d_{j_1,...,j_n},2j)}_{-1},\ldots,-
 {\mathfrak s}^{(d_{j_1,...,j_n},2j)}_{\ell^{(d_{j_1,...,j_n})}_j-2\nu^{(d_{j_1,...,j_n})}_j}) \times\\&\times T(\mathfrak s^{(d_{j_1,...,j_n},2j-1)}_{\nu^{(d_{j_1,...,j_n})}_j-1},\dots,\mathfrak s^{(d_{j_1,...,j_n},2j-1)}_{\ell^{(d_{j_1,...,j_n})}_j}) =
 I_{\ell^{(d_{j_1,...,j_n})}_j-\nu^{(d_{j_1,...,j_n})}_j+2},\end{split}
\end{equation}
where $ \ell_{j}^{(d_{j_1,...,j_n})}=2\nu_N^{(d_{j_1,...,j_n})}-2\mu^{(d_{j_1,...,j_n})}_{j-1}-2$.
\begin{equation}\label{20.4.12}\begin{split}&
     T(l_{\mu^{(d_{j_1\!,\!...\!,\!j_n}\!)}_j\!-\!\nu^{(\!d_{j_1\!,\!...\!,\!j_n}\!)}_j}^{(\!d_{j_1\!,\!...\!,\!j_n},j)},\!\ldots\!,\!l_0^{(\!d_{\!j_1\!,\!...\!,\!j_n\!},\!j\!)},\!-\!{\mathfrak s}^{(d_{j_1,...,j_n},2j+1)}_{0},\ldots,\!-\!
 {\mathfrak s}^{(\!d_{\!j_1\!,...\!,\!j_n},2j\!+\!1)}_{2\!\nu_N^{(\!d_{j_1\!,\!...\!,\!j_n\!})}\!-\!2\mu^{(\!d_{\!j_1\!,\!...\!,\!j_n}\!)}_j\!-\!2})\!\times\\&\times T({\mathfrak s}^{(d_{j_1,...,j_n},2j)}_{\mu^{(d_{j_1,...,j_n})}_j-\nu^{(d_{j_1,...,j_n})}_j-1},\dots,{\mathfrak {s}}^{(d_{j_1,...,j_n},2j)}_{\ell^{(d_{j_1,...,j_n})}_j-2\nu^{(d_{j_1,...,j_n})}_j}) =\\&= I_{2\nu^{(d_{j_1,...,j_n})}_N-\mu^{(d_{j_1,...,j_n})}_j-\nu^{(d_{j_1,...,j_n})}_j},
\end{split}\end{equation}
where 
$\nu^{(d_{j_1,...,j_n})}=\nu^{(d_{j_1,...,j_n})}_j-\mu^{(d_{j_1,...,j_n})}_{j-1}$, $\mu^{(d_{j_1,...,j_n})}=\mu^{(d_{j_1,...,j_n})}_{j}-\nu^{(d_{j_1,...,j_n})}_j$  and $\mathfrak s_i^{(d_{j_1,...,j_n}, 1)}=\mathfrak s_i^{(d_{j_1,...,j_n})}$.
\end {proposition}

\begin{proof}  Let $\mathbf{s}=\{s_{i_1,i_2,...,i_n}\}_{i_1,...,i_n=0}^\ell$  be the sequence of real numbers and let  $j_1,...,j_n \in \mathbb{Z}_{+}$ be fixed, such that at least one of them  wanishes and \eqref{20p.4.1Xqw} holds. Hence, the any solutions of $\bold{MP}\left(\mathbf{s},2\nu^{(d_{j_1,...,j_n})}_N-2\right)$  is associated with the  sequence $\mathfrak{s}^{(d_{j_1,...,j_n})}=\left\{\mathfrak{s}^{(d_{j_1,...,j_n})}_{i}\right\}_{i=0}^{2\nu^{(d_{j_1,...,j_n})}_N-2}$  and admits the following asymptotic expansion
\begin{equation*}\label{20p.4.3.cx3}
	F_{d_{j_1,...,j_n}}(z_1,...,z_n)\!=\!\cfrac{1}{\prod\limits_{i=1}^n \!z_i^{j_i}}\!\left(
	\!-\!\cfrac{\mathfrak{s}^{(d_{j_1,...,j_n})}_{0}}{\prod\limits_{i=1}^n \!z_i}\!-\!\ldots
	\!-\cfrac{\mathfrak{s}^{(d_{j_1,...,j_n})}_{2\nu^{(d_{j_1,...,j_n})}_N\!-\!2}}{\prod\limits_{i=1}^n \!z_i^{2\nu^{(d_{j_1,...,j_n})}_N\!-\!1}}\!+\!o\!\!\left(\!\!\cfrac{1}{\prod\limits_{i=1}^n \!z_i^{2\nu^{(d_{j_1,...,j_n})}_N\!-\!1}
}\!\right)
	\!\!\!\right).
\end{equation*}
By Theorem \ref{20.th3.3},
\[
	\!-\!\cfrac{\mathfrak{s}^{(d_{j_1,...,j_n})}_{0}}{\prod\limits_{i=1}^n \!z_i}\!-\!\ldots
	\!-\cfrac{\mathfrak{s}^{(d_{j_1,...,j_n})}_{2\nu^{(d_{j_1,...,j_n})}_N\!-\!2}}{\prod\limits_{i=1}^n \!z_i^{2\nu^{(d_{j_1,...,j_n})}_N\!-\!1}}\!+\!o\left(\!\!\cfrac{1}{\prod\limits_{i=1}^n \!z_i^{2\nu^{(d_{j_1,...,j_n})}_N\!-\!1}
}\right)
\]
can be represented as
\[\begin{split}
 	&\cfrac{1\big|}{\big|-m_1^{(d_{j_1,...,j_n})}(z_1,..., z_n)\prod\limits_{j=1}^nz_j}+\cfrac{1\big|}{\big|l_1^{(d_{j_1,...,j_n})}(z_1,...,z_n)}+\ldots+\\&+\cfrac{1\big|}{\big| -m_N^{(d_{j_1,...,j_n})}(z_1,..., z_n)\prod\limits_{j=1}^nz_j+\cfrac{1}{\tau(z_1,...,z_n)}},
	\end{split}
\]
where the atom $\left(m_j^{(d_{j_1,...,j_n})} , l^{(d_{j_1,...,j_n})}_j\right)$  can be found by the same formulas as in the Theorem \ref{20.th3.3}, i.e.  \eqref{eq4.9:}--\eqref{20.4.12} hold. Consequently, we obtain  \eqref{20.4.7}--\eqref{20.4.12}. This completes the proof.~\end{proof}

\begin{remark}\label{20p. cor 3.3} If  $d_{j_1,...,j_n}$-associated sequence $\mathfrak s^{(d_{j_1,...,j_n})}=\left\{\mathfrak s^{(d_{j_1,...,j_n})}_i\right\}_{i=0}^{2\nu_N^{(d_{j_1,...,j_n})}-2}$  is the regular sequence, then  all $l_j^{(d_{j_1,...,j_n})}=$const. Furthermore,
 \begin{equation}\label{20.4.14x1}
	l_j^{(d_{j_1,...,j_n})}
   =   \frac{1}{{\mathfrak{s}}_{-1}^{(d_0,2j)}}\qquad\mbox{for all } j=\overline{1,N-1}.
   \end{equation}
\end{remark}

By the similar way we can define the Stieltjes polynomials for $d_{j_1,...,j_n}-$diagonal (see Definition \ref{20p. def3.4}):

\begin{definition}\label{20p. def4.4} (\cite{K17}) The polynomials $P_{j,d_{j_1,...,j_n}}^{+}$ and $Q_{j,d_{j_1,...,j_n}}^{+}$ can be defined  as the solutions of the following system
 \begin{equation}\label{20.4.14xz1}
   \left\{
    \begin{array}{rcl}
    y_{2j,d_{j_1,...,j_n}}-y_{2j-2,d_{j_1,...,j_n}}=l^{(d_{j_1,...,j_n})}_{j}(z_1,...,z_n)y_{2j-1,d_{j_1,...,j_n}},\\
    y_{2j+1,d_{j_1,...,j_n}}-y_{2j-1,d_{j_1,...,j_n}}=-m^{(d_{j_1,...,j_n})}_{j+1}(z_1,...,z_n)y_{2j,d_{j_1,...,j_n}}\prod\limits_{k=1}^nz_k\\
    \end{array}\right.
   \end{equation}
subject to the initial conditions 
 \begin{equation}\label{20.4.14xz2}\begin{split}&
    P_{-1,d_{j_1,...,j_n}}^{+}(z_1,...,z_n)\equiv0\quad\mbox{and
    }\quad P_{0,d_{j_1,...,j_n}}^{+}(z_1,...,z_n)\equiv1,\\&
    Q_{-1,d_{j_1,...,j_n}}^{+}(z_1,...,z_n)\equiv 1\quad\mbox{and
    }\quad Q_{0,d_{j_1,...,j_n}}^{+}(z_1,...,z_n)\equiv0.
    \end{split}
   \end{equation}
   
    $P_{j,d_{j_1,...,j_n}}^{+}$ and $Q_{j,d_{j_1,...,j_n}}^{+}$ are called  $d_{j_1,...,j_n}$ -- Stieltjes polynomials of the first and second kinds, respectively.\end{definition}
\begin{theorem}\label{20p.prop4.7}   Let $\mathbf{s}=\{s_{i_1,i_2,...,i_n}\}_{i_1,...,i_n=0}^{2\nu_N^{(d_{j_1,...,j_n})}-2}$ be the sequence of real numbers   and let $j_1,...,j_n \in \mathbb{Z}_{+}$ be fixed, such that  \eqref{20p.4.1Xqw} holds and $\mathfrak{s}^{(d_{j_1,...,j_n})}=\left\{\mathfrak{s}^{(d_{j_1,...,j_n})}_{i}\right\}_{i=0}^{2n^{(d_{j_1,...,j_n})}_N-1}$ is  $d_{j_1,...,j_n}$-associated sequence of $\mathbf{s}$. Let   $\mathcal{N}\left({\mathfrak s^{(d_{j_1,...,j_n})}}\right)=\left\{\nu_j^{(d_{j_1,...,j_n})}\right\}_{j=1}^N\cup\left\{\mu_j^{(d_{j_1,...,j_n})}\right\}_{j=1}^{N-1}$ be the  set of normal indices of $\mathfrak s^{(d_{j_1,...,j_n})}$. Then:
  \begin{enumerate}
 \item any solution of   $d_{j_1,...,j_n}-$diagonal truncated problem $\bold{MP}\left(\mathbf{s},2\nu^{(d_{j_1,...,j_n})}_N-2\right)$ admits the representation in terms of  $d_{j_1,...,j_n}-$Stieltjes polynomials
 \begin{equation}\label{20.3.prp4.5}\begin{split}
F_{d_{j_1,...,j_n}}&(z_1,...,z_n)\!=\!\cfrac{1}{\prod\limits_{i=1}^n \!\!z_i^{j_i}}\times\\&\times\cfrac{Q^+_{2N-1,d_{j_1,...,j_n}}(z_1,...,z_n)\tau(z_1,...,z_n)+Q^+_{2N-2,d_{j_1,...,j_n}}(z_1,...,z_n)}{P^+_{2N-1,d_{j_1,...,j_n}}\!\!(z_1,...,z_n)\tau(z_1,...,z_n)+P^+_{2N-2,d_{j_1,...,j_n}}\!\!(z_1,...,z_n)},\end{split}
\end{equation}
 where the parameter $\tau$ satisfies \eqref{20.4.8}.

 \item the resolvent matrix $W_{2N-1,d_{j_1,...,j_n}}$   can be represented in terms of  $d_{j_1,...,j_n}-$Stieltjes polynomials
  \begin{equation}\label{20.3.prp4.5.1}\begin{split}
  	&W_{2N-1,d_{j_1,...,j_n}}(z_1,...,z_n)=\\&=\begin{pmatrix}
Q^+_{2N-1,d_{j_1,...,j_n}}(z_1,...,z_n) & Q^+_{2N-2,d_{j_1,...,j_n}}(z_1,...,z_n)  \\
P^+_{2N-1,d_{j_1,...,j_n}}\!\!(z_1,...,z_n)\prod\limits_{i=1}^n\!\! z_i^{j_i} & P^+_{2N-2,d_{j_1,...,j_n}}\!\!(z_1,...,z_n)\prod\limits_{i=1}^n \!\!z_i^{j_i}
\end{pmatrix}\!\!.\end{split}
  \end{equation}
  
  Furthermore,   $W_{2N-1,d_{j_1,...,j_n}}$ admits the  factorization
    \begin{equation}\label{20.3.prp4.5.2}\begin{split}
  	W_{2N-1,d_{j_1,...,j_n}}(z_1,...,z_n)&=A_{d_{j_1,...,j_n}}(z_1,...,z_n)M_{1,d_{j_1,...,j_n}}(z_1,...,z_n)\times\\&\times L_{1,d_{j_1,...,j_n}}(z_1,...,z_n) L_{N-1,d_{j_1,...,j_n}}(z_1,...,z_n)\times\ldots \times\\&\times M_{N,d_{j_1,...,j_n}}(z_1,...,z_n),
  \end{split}  \end{equation}
  where the matrices $A_{d_{j_1,...,j_n}}$, $M_{j,d_{j_1,...,j_n}}$ and $L_{j,d_{j_1,...,j_n}}$ can be found  by
    \begin{equation}\label{20.3.prp4.5.3}\begin{split}&
     A_{d_{j_1,...,j_n}}(z_1,...,z_n)=\begin{pmatrix}
1& 0\\
0& \prod\limits_{i=1}^n z_i
\end{pmatrix},\\&
  M_{j,d_{j_1,...,j_n}}(z_1,...,z_n)=\begin{pmatrix}
1& 0\\
-m_j^{(d_{j_1,...,j_n})}(z_1,...,z_n)\prod\limits_{i=1}^n z_i& 1
\end{pmatrix},\\& L_{j,d_{j_1,...,j_n}}(z_1,...,z_n)=\begin{pmatrix}
1& l_j^{(d_{j_1,...,j_n})}(z_1,...,z_n)\\
0& 1
\end{pmatrix}.\end{split}
  \end{equation}
 \end{enumerate}
 \end{theorem}
 \begin{proof}
 Due to the asymptotic expansion  \eqref{20p.4.3}, any solution of   $d_{j_1,...,j_n}-$diagonal truncated problem  $\bold{MP}\left(\mathbf{s},2\nu^{(d_{j_1,...,j_n})}_N-2\right)$ admits the following asymptotic expansion
 \begin{equation}\label{20p.4.3x.1}\begin{split}
	F_{d_{j_1,\!...,\!j_n}}\!\!(z_1\!,...\!,\!z_n)&\!\!=\!\!\cfrac{1}{\prod\limits_{i=1}^n \!
\!z_i^{j_i}}\!\! \!\left(\!\!
	-\cfrac{\mathfrak{s}^{(d_{j_1\!,\!...\!,\!j_n})}_{0}}{\prod\limits_{i=1}^n \!\!z_i}\!-\!\ldots\!
	-\!\cfrac{\mathfrak{s}^{(d_{j_1,...,j_n})}_{2\nu^{(d_{j_1\!,\!...\!,\!j_n})}_N-2}}{\prod\limits_{i=1}^n \!\!z_i^{2\nu^{(d_{j_1\!,\!...\!,\!j_n})}_N\!-\!1}}\!+\!o\!\!\left(\cfrac{1}{\prod\limits_{i=1}^n \!\!z_i^{2\nu^{(d_{j_1\!,\!...\!,\!j_n})}_N\!-\!1}
}\!\!\right)
	\!\!\right).
\end{split}\end{equation}

On the other hand, we can consider the expression in the bracket of  \eqref{20p.4.3x.1},  i.e.
 \begin{equation}\label{20p.4.3x.1.c}
\tilde {F}_{d_{j_1,\!...,\!j_n}}\!\!(z_1\!,...\!,\!z_n)=-\cfrac{\mathfrak{s}^{(d_{j_1\!,\!...\!,\!j_n})}_{0}}{\prod\limits_{i=1}^n \!\!z_i}\!-\!\ldots\!
	-\!\cfrac{\mathfrak{s}^{(d_{j_1,...,j_n})}_{2\nu^{(d_{j_1\!,\!...\!,\!j_n})}_N-2}}{\prod\limits_{i=1}^n \!\!z_i^{2\nu^{(d_{j_1\!,\!...\!,\!j_n})}_N\!-\!1}}\!+\!o\!\!\left(\cfrac{1}{\prod\limits_{i=1}^n \!\!z_i^{2\nu^{(d_{j_1\!,\!...\!,\!j_n})}_N\!-\!1}
}\right).
\end{equation}

By  Proposition \ref{20p.prop3.7}, \eqref{20p.4.3x.1.c} can be represented in the following form

 \begin{equation}\label{20p.4.3x.1.c2}
\tilde {F}_{d_{j_1,\!...,\!j_n}}\!\!(z_1\!,...\!,\!z_n)\!=\!\cfrac{Q^+_{2N\!-\!1,d_{j_1,...,j_n}}\!(z_1,...,z_n)\tau\!(z_1,...,z_n)\!+\!Q^+_{2N\!-\!2,d_{j_1,...,j_n}}\!(z_1,...,z_n)}{P^+_{2N\!-\!1,d_{j_1,...,j_n}}\!\!(z_1,...,z_n)\tau(z_1,...,z_n)\!+\!P^+_{2N\!-\!2,d_{j_1,...,j_n}}\!\!(z_1,...,z_n)},
\end{equation}
where $P^+_{i,d_{j_1,...,j_n}}$ and $Q^+_{i,d_{j_1,...,j_n}}$ are defined by \eqref{20.4.14xz1}--\eqref{20.4.14xz2} and the paramiter $\tau$ satisfies \eqref{20.4.8}. The resolvent matrix 
$\tilde{W}_{2N-1,d_{j_1,...,j_n}}$ associated with $\tilde {F}_{d_{j_1,\!...,\!j_n}}$ can be found by
  \begin{equation}\label{20.3.prp4.5.1c1}
  	\tilde{W}_{2N-1,d_{j_1,...,j_n}}\!\!(z_1,...,z_n)\!=\!\!\begin{pmatrix}
Q^+_{2N-1,d_{j_1,...,j_n}}\!(z_1,...,z_n)\! &\! Q^+_{2N-2,d_{j_1,...,j_n}}\!(z_1,...,z_n)  \\
P^+_{2N-1,d_{j_1,...,j_n}}\!\!(z_1,...,z_n)\!&\! P^+_{2N-2,d_{j_1,...,j_n}}\!\!(z_1,...,z_n)
\end{pmatrix}\!\!.
  \end{equation}

By Proposition \ref{20p.prop3.7}, $\tilde{W}_{2N-1,d_{j_1,...,j_n}}$ admits the following factorization 
 \begin{equation}\label{20.3.prp4.5.1c2}\begin{split}
		\tilde{W}_{2N-1,d_{j_1,...,j_n}}\!\!(z_1,...,z_n)&=M_{1,d_{j_1,...,j_n}}(z_1,...,z_n)\times\\&\times L_{1,d_{j_1,...,j_n}}(z_1,...,z_n) L_{N-1,d_{j_1,...,j_n}}(z_1,...,z_n)\times\ldots \times\\&\times M_{N,d_{j_1,...,j_n}}(z_1,...,z_n),
  \end{split} 
  \end{equation}
 where the matrices $M_{j,d_{j_1,...,j_n}}$ and $L_{j,d_{j_1,...,j_n}}$ are defined by \eqref{20.3.prp4.5.3}.
 
 Due to
\[
	 {F}_{d_{j_1,\!...,\!j_n}}\!(z_1\!,...\!,\!z_n)=\cfrac{1}{\prod\limits_{i=1}^n \!
\!z_i^{j_i}} \tilde {F}_{d_{j_1,\!...,\!j_n}}\!(z_1\!,...\!,\!z_n)
\]
 \eqref{20p.4.3x.1.c2} and \eqref{20.3.prp4.5.1c1}, we obtain \eqref{20.3.prp4.5} and \eqref{20.3.prp4.5.1}, respectively. Moreover, if we set
 \[
  A_{d_{j_1,...,j_n}}(z_1,...,z_n)=\begin{pmatrix}
1& 0\\
0& \prod\limits_{i=1}^n z_i
\end{pmatrix},
 \]
 then 
 \[
 	 A_{d_{j_1,...,j_n}}(z_1,...,z_n)\tilde{W}_{2N-1,d_{j_1,...,j_n}}\!\!(z_1,...,z_n)={W}_{2N-1,d_{j_1,...,j_n}}\!\!(z_1,...,z_n)
 \]
 and the factorization \eqref{20.3.prp4.5.2}--\eqref{20.3.prp4.5.3} holds.   This completes the proof.~\end{proof} 

 \subsection{Even $d_{j_1,...,j_n}$-diagonal  truncated problem} The next, we study the even $d_{j_1,...,j_n}$-diagonal  truncated  problem.
 
\begin{theorem}\label{20.th4.2} Let $\mathbf{s}=\{s_{i_1,i_2,...,i_n}\}_{i_1,...,i_n=0}^{2n^{(d_{j_1,...,j_n})}_N-1}$ be a sequence of real numbers   and let $j_1,...,j_n \in \mathbb{Z}_{+}$ be fixed such that  \eqref{20p.4.1Xqw} holds and  $\mathfrak{s}^{(d_{j_1,...,j_n})}=\left\{\mathfrak{s}^{(d_{j_1,...,j_n})}_{i}\right\}_{i=0}^{2n^{(d_{j_1,...,j_n})}_N-1}$ is the $d_{j_1,...,j_n}$-associated sequence of $\mathbf{s}$. Let   $\mathcal{N}({\mathfrak s^{(d_{j_1,...,j_n})}})$ be the set of normal indices  of the sequence  $\mathfrak s^{(d_{j_1,...,j_n})}$ and $n^{(d_{j_1,...,j_n})}_N\in \mathcal{N}({\mathfrak s^{(d_{j_1,...,j_n})}})$. Then any solution of $d_{j_1,...,j_n}$-diagonal truncated problem $\bold{MP}(\mathbf{s},2n^{(d_{j_1,...,j_n})}_N-1)$ admits the following  representation
\begin{equation}\label{20.4.4}\begin{split}
	F_{d_{j_1,...,j_n}}(z_1,...z_n)\!=\!&-\cfrac{1}{\prod\limits_{i=1}^n z_i^{j_i}}\cdot\cfrac{b_0^{(d_{j_1,...,j_n})}\big|}{\big|a_0^{(d_{j_1,...,j_n})}(z_1,...,z_n)}\!-\!\cfrac{b_1^{(d_{j_1,...,j_n})}\big|}{\big|a_1^{(d_{j_1,...,j_n})}(z_1,...,z_n)}\!-\!\ldots\!-\!\\&
	-\cfrac{b_{N-1}^{(d_{j_1,...,j_n})}\big|}{\big|a_{N-1}^{(d_{j_1,...,j_n})}(z_1,...,z_n)+\tau(z_1,...,z_n)}.
\end{split}\end{equation}
where the parameter $\tau$ satisfies the following
\begin{equation}\label{20.4.5}
	\tau(z_1,...,z_n)=o(1).
\end{equation}
 The atoms $\left(a^{(d_{j_1,...,j_n})}_i,b^{(d_{j_1,...,j_n})}_i\right)$ can be  found by
 \begin{equation}\label{20.4.6}\begin{split}
	&b^{(d_{j_1,...,j_n})}_0=\mathfrak s^{(d_{j_1,...,j_n})}_{n_1-1}\quad  \mbox{and}\quad b^{(d_{j_1,...,j_n})}_j=\mathfrak s^{(d_{j_1,...,j_n}, j)}_{n^{d_{j_1,...,j_n}}_{j}-n^{(d_{j_1,...,j_n})}_{j-1}-1}\\&
	 a^{(\!d_{j_1,...,j_n}\!)}_{j}(\!z_1,...,z_n\!)\!=\!\!\cfrac{1}{D^{(\!d_{j_1,...,j_n})\!}_{\nu}} \!\begin{vmatrix}
	\mathfrak  s^{(\!d_{j_1,...,j_n}, j-1\!)}_{0}&\mathfrak  s^{(\!d_{j_1,...,j_n}, j-1\!)}_{1}& \ldots&  \mathfrak s^{(\!d_{j_1,...,j_n}, j-1\!)}_{\nu}\\
	  \ldots& \ldots&\ldots& \ldots\\
	   \mathfrak s^{(\!d_{j_1,...,j_n}, j-1\!)}_{\nu-1}& \mathfrak s^{(\!d_{j_1,...,j_n}, j-1\!)}_{\nu}& \ldots&\mathfrak  s^{(\!d_{j_1,...,j_n}, j-1\!)}_{2\nu-1}\\
	    1& \prod\limits_{i=1}^n z_i&\ldots& \prod\limits_{i=1}^n z^{\nu}_i
	 \end{vmatrix}
	 \!\!,\\&
	\mathfrak s^{(\!d_{j_1,...,j_n},\! j\!)}_i\!\!=\!\!\cfrac{(\!-\!1)^{i\!+\!\nu}}{\left(\!\mathfrak s^{(\!d_{j_1,...,j_n}, j\!-\!1\!)}_{\nu\!-\!1}\!\right)^{i\!+\!\nu\!+\!2}}\begin{vmatrix}
\!\mathfrak s^{(\!d_{j_1,...,j_n}, \!j\!-\!1\!)}_{\nu}\!& \!\mathfrak s^{(\!d_{j_1,...,j_n},\! j\!-\!1\!)}_{\nu\!-\!1}\! &0&\ldots& 0\\
\vdots& \ddots &\ddots&\ddots & \vdots\\
\vdots&  &\ddots&\ddots &0\\
\!\vdots\!&  \!\!&\! &\!\ddots \! \!&\!\!\mathfrak  s^{(\!d_{j_1,...,j_n},\! j\!-\!1\!)}_{\nu\!-\!1}\!\\
\mathfrak s^{(\!d_{j_1\!,...,\!j_n},\!j\!-\!1\!)}_{2\nu+i}\!&\! \!\ldots\! \!&\!\ldots\!&\!\ldots\!&\! \mathfrak s^{(\!d_{j_1,...,j_n}, \!j\!-\!1\!)}_{\nu}
\end{vmatrix}\!\!,
\end{split}
\end{equation}
 $i=\overline{0,2n^{(d_{j_1,...,j_n})}_N-2n^{(d_{j_1,...,j_n})}_j-1}$, $j=\overline {0, N}$, $\mathfrak s^{(d_{j_1,...,j_n}, 0)}_{\nu}=\mathfrak s^{(d_{j_1,...,j_n})}_{\nu}$, $n^{(d_{j_1,...,j_n})}_{0}=0$ and  $\nu=n^{(d_{j_1,...,j_n})}_j-n^{(d_0)}_{j-1}$.
\end{theorem}

\begin{proof} The proof  is based on Proposition~\ref{20.th3.1}. According to representation \eqref{20p.4.3}, the Proposition~\ref{20.th3.1} can be applied to  
\[
	-\cfrac{\mathfrak{s}^{(d_{j_1,...,j_n})}_{0}}{\prod\limits_{i=1}^n z_i}-\ldots
	-\cfrac{\mathfrak{s}^{(d_{j_1,...,j_n})}_{2n^{(d_{j_1,...,j_n})}_N-1}}{\prod\limits_{i=1}^n z_i^{2n^{(d_{j_1,...,j_n})}_N}}+o\left(\cfrac{1}{\prod\limits_{i=1}^n z_i^{2n^{(d_{j_1,...,j_n})}_N}
}\right)
\]
and we obtain \eqref{20.4.4}--\eqref{20.4.6}.  This completes the proof.~\end{proof}

\begin{proposition}\label{20.th4.4} Let $\mathbf{s}=\{s_{i_1,i_2,...,i_n}\}_{i_1,...,i_n=0}^{2\mu_N^{(d_{j_1,...,j_n})}-1}$ be the sequence of real numbers   and let $j_1,...,j_n \in \mathbb{Z}_{+}$ be fixed such that \eqref{20p.4.1Xqw} holds and  $\mathfrak{s}^{(d_{j_1,...,j_n})}=\left\{\mathfrak{s}^{(d_{j_1,...,j_n})}_{i}\right\}_{i=0}^{2\mu^{(d_{j_1,...,j_n})}_N-1}$ is the $d_{j_1,...,j_n}$-associated sequence of $\mathbf{s}$. Let   $\mathcal{N}\!\!\left({\!\mathfrak s^{(d_{j_1,...,j_n})}}\!\right)\!\!=\left\{\nu_j^{(d_{j_1,...,j_n})}\right\}_{j=1}^N\cup\left\{\mu_j^{(d_{j_1,...,j_n})}\right\}_{j=1}^{N}$ be the  set of normal indices of $\mathfrak s^{(d_{j_1,...,j_n})}$. Then any solution of  $d_{j_1,...,j_n}$-diagonal truncated problem $\bold{MP}(\mathbf{s},2\mu^{(d_{j_1,...,j_n})}_N-1)$ admits the representation
 \begin{equation}\label{20.4.13}\begin{split}
 	F_{d_{j_1,...,j_n}}\!(\!z_1\!,..., \!z_n\!)\!=\!&\cfrac{1}{\prod\limits_{i=1}^n\!\! z_i^{j_i}}\!\cdot\!\!\cfrac{1\big|}{\big|\!-\!m_1^{(d_{j_1,...,j_n})}\!\!(z_1,..., z_n)\!\!\prod\limits_{j=1}^n\!\!z_j}\!+\!\!\cfrac{1\big|}{\big|l_1^{(d_{j_1,...,j_n}\!)}\!(z_1,...,z_n)}\!+\!\ldots\!+\!\\&+\cfrac{1\big|}{\big| -m_N^{(d_{j_1,...,j_n})}(z_1,..., z_n)\prod\limits_{j=1}^nz_j+\cfrac{1}{\tau(z_1,...,z_n)}}+\\&
	+\cfrac{1\big|}{\big| l_N^{(d_{j_1,...,j_n})}(z_1,..., z_n)\prod\limits_{j=1}^nz_j+\tau(z_1,...,z_n)},
	\end{split}
 \end{equation}
 where  the parameter $\tau$ satisfies the following  
\begin{equation}\label{20.4.15}
  	\tau(z_1, ..., z_n)=o\left(1\right),
   \end{equation}
 the atoms $\left(m_j^{(d_{j_1,...,j_n})} , l^{(d_{j_1,...,j_n})}_j\right)$ can be found by~\eqref{eq4.9:}--\eqref{20.4.12}.
\end{proposition}

\begin{proof}
Assume $\mathbf{s}=\{s_{i_1,i_2,...,i_n}\}_{i_1,...,i_n=0}^{2\mu_N^{(d_{j_1,...,j_n})}-1}$ is the sequence of real numbers   and let $j_1,...,j_n \in \mathbb{Z}_{+}$ be fixed such that \eqref{20p.4.1Xqw} holds and  $\mathfrak{s}^{(d_{j_1,...,j_n})}=\left\{\mathfrak{s}^{(d_{j_1,...,j_n})}_{i}\right\}_{i=0}^{2\mu^{(d_{j_1,...,j_n})}_N-1}$ is the $d_{j_1,...,j_n}$-associated sequence of $\mathbf{s}$, where $\mu^{(d_{j_1,...,j_n})}_N$ is the largest normal index of the sequence $\mathfrak{s}^{(d_{j_1,...,j_n})}$. Then the associated function $F_{d_{j_1,\!...,\!j_n}}$ takes the following 
 \begin{equation}\label{20p.4.3x.tr}\begin{split}
	F_{d_{j_1,\!...,\!j_n}}\!\!(z_1\!,...\!,\!z_n)&\!\!=\!\!\cfrac{1}{\prod\limits_{i=1}^n \!
\!z_i^{j_i}}\!\! \!\left(\!\!
	-\cfrac{\mathfrak{s}^{(d_{j_1\!,\!...\!,\!j_n})}_{0}}{\prod\limits_{i=1}^n \!\!z_i}\!-\!\ldots\!
	-\!\cfrac{\mathfrak{s}^{(d_{j_1,...,j_n})}_{2\mu^{(d_{j_1\!,\!...\!,\!j_n})}_N-1}}{\prod\limits_{i=1}^n \!\!z_i^{2\mu^{(d_{j_1\!,\!...\!,\!j_n})}_N\!}}\!+\!o\!\!\left(\cfrac{1}{\prod\limits_{i=1}^n \!\!z_i^{2\mu^{(d_{j_1\!,\!...\!,\!j_n})}_N\!}
}\!\!\right)
	\!\!\right).
\end{split}\end{equation}

Setting 
 \begin{equation}\label{20.4.13p1}
	\tilde{F}_{d_{j_1,\!...,\!j_n}}\!\!(z_1\!,...\!,\!z_n)=-\cfrac{\mathfrak{s}^{(d_{j_1\!,\!...\!,\!j_n})}_{0}}{\prod\limits_{i=1}^n \!\!z_i}\!-\!\ldots\!
	-\!\cfrac{\mathfrak{s}^{(d_{j_1,...,j_n})}_{2\mu^{(d_{j_1\!,\!...\!,\!j_n})}_N-1}}{\prod\limits_{i=1}^n \!\!z_i^{2\mu^{(d_{j_1\!,\!...\!,\!j_n})}_N\!}}\!+\!o\!\!\left(\cfrac{1}{\prod\limits_{i=1}^n \!\!z_i^{2\mu^{(d_{j_1\!,\!...\!,\!j_n})}_N\!},
}\!\!\right)
  \end{equation}
we get the following relation between $F_{d_{j_1,\!...,\!j_n}}$ and $\tilde{F}_{d_{j_1,\!...,\!j_n}}$, i.e.
 \begin{equation}\label{20.4.13p2}
	F_{d_{j_1,\!...,\!j_n}}\!\!(z_1\!,...\!,\!z_n)=\cfrac{1}{\prod\limits_{i=1}^n \!
\!z_i^{j_i}}\tilde{F}_{d_{j_1,\!...,\!j_n}}\!\!(z_1\!,...\!,\!z_n).
  \end{equation}
By Proposition~\ref{20.th3.4}, $\tilde{F}_{d_{j_1,\!...,\!j_n}}$ can be represented as
\begin{equation}\label{20.4.13p.r1}\begin{split}
 	\tilde{F}_{d_{j_1,...,j_n}}\!(\!z_1\!,..., \!z_n\!)\!=\!&\cfrac{1\big|}{\big|\!-\!m_1^{(d_{j_1,...,j_n})}\!\!(z_1,..., z_n)\!\!\prod\limits_{j=1}^n\!\!z_j}\!+\!\!\cfrac{1\big|}{\big|l_1^{(d_{j_1,...,j_n}\!)}\!(z_1,...,z_n)}\!+\!\ldots\!+\!\\&+\cfrac{1\big|}{\big| -m_N^{(d_{j_1,...,j_n})}(z_1,..., z_n)\prod\limits_{j=1}^nz_j+\cfrac{1}{\tau(z_1,...,z_n)}}+\\&
	+\cfrac{1\big|}{\big| l_N^{(d_{j_1,...,j_n})}(z_1,..., z_n)\prod\limits_{j=1}^nz_j+\tau(z_1,...,z_n)},
	\end{split}
 \end{equation}
 where  the parameter $\tau$ satisfies \eqref{20.4.15} and 
 the atoms $\left(m_j^{(d_{j_1,...,j_n})} , l^{(d_{j_1,...,j_n})}_j\right)$ can be found by~\eqref{eq4.9:}--\eqref{20.4.12}. Due to \eqref{20.4.13p2} and \eqref{20.4.13p.r1}, we obtain \eqref{20.4.13}.  This completes the proof.~\end{proof}

\begin{remark}\label{20p. cor 4x.5} If  $d_{j_1,...,j_n}$-associated sequence $\mathfrak s^{(d_{j_1,...,j_n})}=\left\{\mathfrak s^{(d_{j_1,...,j_n})}_i\right\}_{i=0}^{2\mu_N^{(d_{j_1,...,j_n})}-1}$  is  regular, then $l_j^{(d_{j_1,...,j_n})}=$const and one can be calculated by
 \begin{equation}\label{20.4x.14y1}
	l_j^{(d_{j_1,...,j_n})}
   =   \frac{1}{{\mathfrak{s}}_{-1}^{(d_{j_1,...,j_n},2j)}}\qquad\mbox{for all } j=\overline{1,N}.
   \end{equation}
\end{remark}

 \begin{theorem}\label{20p.prop4.8}  Let $\mathbf{s}=\{s_{i_1,i_2,...,i_n}\}_{i_1,...,i_n=0}^{2\mu_N^{(d_{j_1,...,j_n})}-1}$ be the sequence of real numbers   and let $j_1,...,j_n \in \mathbb{Z}_{+}$ be fixed such that \eqref{20p.4.1Xqw} holds and   $\mathfrak{s}^{(d_{j_1,...,j_n})}=\left\{\mathfrak{s}^{(d_{j_1,...,j_n})}_{i}\right\}_{i=0}^{2n^{(d_{j_1,...,j_n})}_N-1}$ is  $d_{j_1,...,j_n}$-associated sequence of $\mathbf{s}$. Let   $\mathcal{N}\left({\mathfrak s^{(d_{j_1,...,j_n})}}\right)=\left\{\nu_j^{(d_{j_1,...,j_n})}\right\}_{j=1}^N\cup\left\{\mu_j^{(d_{j_1,...,j_n})}\right\}_{j=1}^{N}$ be the  set of normal indices of $\mathfrak s^{(d_{j_1,...,j_n})}$. Then:
   \begin{enumerate}
 \item Any solution of  $d_{j_1,...,j_n}$--diagonal truncated problem $\bold{MP}\left(\mathbf{s},2\mu^{(d_{j_1,...,j_n})}_N-1\right)$ admits the representation in terms of $d_{j_1,...,j_n}-$Stieltjes polynomials
 \begin{equation}\label{20.3.prp4.8}\begin{split}
F_{d_{j_1,...,j_n}}&(z_1,...,z_n)=\cfrac{1}{\prod\limits_{i=1}^n z_i^{j_i}}\times\\&\times\cfrac{Q^+_{2N-1,d_{j_1,...,j_n}}(z_1,...,z_n)\tau(z_1,...,z_n)+Q^+_{2N,d_{j_1,...,j_n}}(z_1,...,z_n)}{P^+_{2N-1,d_{j_1,...,j_n}}(z_1,...,z_n)\tau(z_1,...,z_n)+P^+_{2N,d_{j_1,...,j_n}}(z_1,...,z_n)},\end{split}
\end{equation}
 where the parameter $\tau$ satisfies \eqref{20.4.15}.
 \item $W_{2N,d_{j_1,...,j_n}}$ can be represented in terms of the  $d_{j_1,...,j_n}-$Stieltjes polynomials
  \begin{equation}\label{20.3.prp4.8.1}\begin{split}
  	W_{2N,d_{j_1,...,j_n}}&(z_1,...,z_n)=\\&=\begin{pmatrix}
Q^+_{2N-1,d_{j_1,...,j_n}}(z_1,...,z_n) & Q^+_{2N,d_{j_1,...,j_n}}(z_1,...,z_n)  \\
P^+_{2N-1,d_{j_1,...,j_n}}(z_1,...,z_n) \prod\limits_{i=1}^n \!\!z_i^{j_i}& P^+_{2N,d_{j_1,...,j_n}}(z_1,...,z_n)\prod\limits_{i=1}^n \!\!z_i^{j_i}
\end{pmatrix}\!\!.\end{split}
  \end{equation}
  Furthermore,  the resolvent matrix $W_{2N,d_{j_1,...,j_n}}$ admits the following factorization
    \begin{equation}\label{20.3.prp4.8.2}\begin{split}
  	W_{2N,d_{j_1,...,j_n}}(z_1,...,z_n)&=A_{d_{j_1,...,j_n}}(z_1,...,z_n)M_{1,d_{j_1,...,j_n}}(z_1,...,z_n)\times\\&\times L_{1,d_{j_1,...,j_n}}(z_1,...,z_n)\times\ldots \times L_{N-1,d_{j_1,...,j_n}}(z_1,...,z_n)\times\\&\\&\times M_{N,d_{j_1,...,j_n}}(z_1,...,z_n)L_{N,d_{j_1,...,j_n}}(z_1,...,z_n),
  \end{split}  \end{equation}
  where the matrices $A_{d_{j_1,...,j_n}}$, $M_{j,d_{j_1,...,j_n}}$ and $L_{j,d_{j_1,...,j_n}}$ are defined by \eqref{20.3.prp4.5.3}.
 \end{enumerate}
 \end{theorem}
 
 \begin{proof} Let the assumptions of the Theorem hold.  Any solution of  $d_{j_1,...,j_n}$--diagonal truncated problem $\bold{MP}\left(\mathbf{s},2\mu^{(d_{j_1,...,j_n})}_N-1\right)$ takes expansion \eqref{20p.4.3x.tr}.  In a similar way we define $\tilde{F}_{d_{j_1,...,j_n}}$ by \eqref{20.4.13p1}.  By Proposition \ref{20p.prop3.8},  $\tilde{F}_{d_{j_1,...,j_n}}$ can be represented in the terms of  Stieltjes polynomials
\[
\tilde{F}_{d_{j_1,...,j_n}}(z_1,...,z_n)=\cfrac{Q^+_{2N-1,d_{j_1,...,j_n}}(z_1,...,z_n)\tau(z_1,...,z_n)+Q^+_{2N,d_{j_1,...,j_n}}(z_1,...,z_n)}{P^+_{2N-1,d_{j_1,...,j_n}}(z_1,...,z_n)\tau(z_1,...,z_n)+P^+_{2N,d_{j_1,...,j_n}}(z_1,...,z_n)},
\]
where $P^+_{j,d_{j_1,...,j_n}}$ and $Q^+_{j,d_{j_1,...,j_n}}$ are defined by \eqref{20.4.14xz1}--\eqref{20.4.14xz2}, the parameter $\tau$ satisfies  \eqref{20.4.15}. Furthermore, the solution matrix $\tilde{W}_{2N,d_{j_1,...,j_n}}$ associated with the $\tilde{F}_{d_{j_1,...,j_n}}$ takes the following form
 \begin{equation}\label{20.3.prp4.8.1b4}
  \tilde{W}_{2N,d_{j_1,...,j_n}}(z_1,...,z_n)=\begin{pmatrix}
Q^+_{2N-1,d_{j_1,...,j_n}}(z_1,...,z_n) & Q^+_{2N,d_{j_1,...,j_n}}(z_1,...,z_n)  \\
P^+_{2N-1,d_{j_1,...,j_n}}(z_1,...,z_n) & P^+_{2N,d_{j_1,...,j_n}}(z_1,...,z_n)
\end{pmatrix}\!\!
  \end{equation}
and $\tilde{W}_{2N,d_{j_1,...,j_n}}$ admits the following factorization
 \begin{equation}\label{20.3.prp4.8.2xc4}\begin{split}
  	\tilde{W}_{2N,d_{j_1,...,j_n}}(z_1,...,z_n)&=M_{1,d_{j_1,...,j_n}}(z_1,...,z_n)\times\\&\times L_{1,d_{j_1,...,j_n}}(z_1,...,z_n)\times\ldots\times L_{N-1,d_{j_1,...,j_n}}(z_1,...,z_n) \times\\&\\&\times M_{N,d_{j_1,...,j_n}}(z_1,...,z_n)L_{N,d_{j_1,...,j_n}}(z_1,...,z_n),
  \end{split}  \end{equation}
 where the matrices   $M_{j,d_{j_1,...,j_n}}$ and $L_{j,d_{j_1,...,j_n}}$ are defined by \eqref{20.3.prp4.5.3}.
 
  Due to $F_{d_{j_1,\!...,\!j_n}}=\frac{1}{\prod\limits_{i=1}^n \!
\!z_i^{j_i}}\tilde{F}_{d_{j_1,\!...,\!j_n}}$, we obtain \eqref{20.3.prp4.8} and \eqref{20.3.prp4.8.1}. Let us set 
  \begin{equation}\label{Ax1}
  A_{d_{j_1,...,j_n}}(z_1,...,z_n)=\begin{pmatrix}
1& 0\\
0& \prod\limits_{i=1}^n z_i
\end{pmatrix},
 \end{equation}
then ${W}_{2N,d_{j_1,...,j_n}}=  A_{d_{j_1,...,j_n}}\tilde{W}_{2N,d_{j_1,...,j_n}}$, i.e. the factorization \eqref{20.3.prp4.8.2} holds. This completes the proof.~\end{proof}


\subsection{$d_{j_1,...,j_n}$-diagonal full problem}

\begin{proposition}\label{20pxprop4.2} Let  $\mathbf{s}=\{s_{i_1,i_2,...,i_n}\}_{i_1,...,i_n=0}^{\infty}$ be the sequence of real numbers 
and let $j_1,...,j_n \in \mathbb{Z}_{+}$ be fixed such that \eqref{20p.4.1Xqw} holds and  $\mathfrak{s}^{(d_{j_1,...,j_n})}=\left\{\mathfrak{s}^{(d_{j_1,...,j_n})}_{i}\right\}_{i=0}^{\infty}$ is  $d_{j_1,...,j_n}$-associated sequence of $\mathbf{s}$. Let  $\mathfrak s^{(d_{j_1,...,j_n})}=\left\{\mathfrak s^{(d_{j_1,...,j_n})}_i\right\}_{i=0}^{\infty}$ be  $d_{j_1,...,j_n}$-associated sequence of $\mathbf{s}$ and   regular, let $\mathcal{N}\left({\mathfrak s^{(d_{j_1,...,j_n})}}\right)$ be the set of normal indices  of $\mathfrak s^{(d_{j_1,...,j_n})}$. Assume additionally that $l_j^{(d_{j_1,...,j_n})}>0$ for all $j\in \mathbb{N}$.
Then:
   \begin{enumerate}
 \item Any solution of  $d_{j_1,...,j_n}$-diagonal full problem $\bold{MP}(\mathbf{s})$ admits the representation
\begin{equation}\label{20.4.eert_1}
	F_{d_{j_1,...,j_n}}(z_1,...z_n)=\cfrac{1}{\prod\limits_{i=1}^n z_i^{j_i}}\cdot\underset{i=0}{\overset{\infty}{\mathbf{K}}}\left(-\cfrac{b_{i}^{(d_{j_1,...,j_n})}}{a_{i}^{(d_{j_1,...,j_n})}(z_1,...,z_n)}\right),
	\end{equation}
where the atoms $\left(a_i^{(d_{j_1,...,j_n})}, b_i^{(d_{j_1,...,j_n})}\right)$ are defined by~\eqref{20.4.6}.

 \item $d_{j_1,...,j_n}$-diagonal full problem  $\bold{MP}(\mathbf{s})$  is indeterminate if and only if 
 \begin{equation}\label{20.4.eert_2}
 	\sum_{i=0}^{\infty}\!\!\cfrac{\left|P_{i}^{(d_{j_1,...,j_n})}(0,...,0)\right|^2}{\prod\limits_{j=0}^i b_i^{(d_{j_1,...,j_n})}}<\infty\quad\mbox{and}\quad 
	\sum_{i=0}^{\infty}\!\!\cfrac{\left|Q_{i}^{(d_{j_1,...,j_n})}(0,...,0)\right|^2}{\prod\limits_{j=0}^i b_i^{(d_{j_1,...,j_n})}}<\infty.
	\end{equation}
\end{enumerate}

\end{proposition}
\begin{proof}
Let the assumptions of the Proposition hold. Then any solution $F_{d_{j_1,...,j_n}}$ of  $d_{j_1,...,j_n}$-diagonal full problem $\bold{MP}(\mathbf{s})$ admits asymptotic expansion
 \begin{equation}\label{Fx1}
F_{d_{j_1,...,j_n}}(z_1,...,z_n)=-\cfrac{1}{\prod\limits_{i=1}^n \!
\!z_i^{j_i}}
	\sum_{k=1}^{\infty}\cfrac{\left(\mathfrak{s}^{(d_{j_1\!,\!...\!,\!j_n})}_{0}\right)^k}{\prod\limits_{i=1}^n \!\!z^k_i}.
\end{equation}

In the next step, we define the function $\tilde{F}_{d_{j_1,...,j_n}}$ such that
 \begin{equation}\label{Fx2}
\tilde{F}_{d_{j_1,...,j_n}}(z_1,...,z_n)=-\sum_{k=1}^{\infty}\cfrac{\left(\mathfrak{s}^{(d_{j_1\!,\!...\!,\!j_n})}_{0}\right)^k}{\prod\limits_{i=1}^n \!\!z^k_i}.
	\end{equation}

By Proposition \ref{20pxprop3.2},  $\tilde{F}_{d_{j_1,...,j_n}}$  takes the following form
\[
\tilde{F}_{d_{j_1,...,j_n}}(z_1,...,z_n)=\underset{i=0}{\overset{\infty}{\mathbf{K}}}\left(-\cfrac{b_{i}^{(d_{j_1,...,j_n})}}{a_{i}^{(d_{j_1,...,j_n})}(z_1,...,z_n)}\right),
	\]
where the atoms $\left(a_i^{(d_{j_1,...,j_n})}, b_i^{(d_{j_1,...,j_n})}\right)$ can be found by~\eqref{20.4.6} and a full problem  $\bold{MP}(\mathfrak{s}^{(d_{j_1\!,\!...\!,\!j_n})})$  is indeterminate if and only if \eqref{20.4.eert_2} holds. Consequently, $F_{d_{j_1,...,j_n}}$ admits the representation \eqref{20.4.eert_1} and   $d_{j_1,...,j_n}$-diagonal full problem  $\bold{MP}(\mathbf{s})$  is indeterminate if and only if \eqref{20.4.eert_2} holds.  This completes the proof.~\end{proof}

\begin{theorem}\label{20.th4.11} Let  $\mathbf{s}=\{s_{i_1,i_2,...,i_n}\}_{i_1,...,i_n=0}^{\infty}$ be the sequence of real numbers 
and let $j_1,...,j_n \in \mathbb{Z}_{+}$ be fixed such that \eqref{20p.4.1Xqw} holds and  $\mathfrak{s}^{(d_{j_1,...,j_n})}=\left\{\mathfrak{s}^{(d_{j_1,...,j_n})}_{i}\right\}_{i=0}^{\infty}$ is  $d_{j_1,...,j_n}$-associated sequence of $\mathbf{s}$. Let  $\mathfrak s^{(d_{j_1,...,j_n})}=\left\{\mathfrak s^{(d_{j_1,...,j_n})}_i\right\}_{i=0}^{\infty}$ be the $d_{j_1,...,j_n}$-associated sequence of $\mathbf{s}$ and   regular, let $\mathcal{N}\left({\mathfrak s^{(d_{j_1,...,j_n})}}\right)$ be the set of normal indices  of $\mathfrak s^{(d_{j_1,...,j_n})}$.  Then:
   \begin{enumerate}
 \item Any solution of  $d_{j_1,...,j_n}$-diagonal full problem $\bold{MP}(\mathbf{s})$ admits the representation
\begin{equation}\label{20.4.289}
	F_{d_{j_1,...,j_n}}(z_1,...z_n)=\cfrac{1}{\prod\limits_{i=1}^n z_i^{j_i}}\cdot\underset{i=1}{\overset{\infty}{\mathbf{K}}}\left(\cfrac{1}{-m_i^{(d_{j_1,...,j_n})}(z_1,..., z_n)\prod\limits_{j=1}^nz_j+\cfrac{1}{l_i^{(d_{j_1,...,j_n})}}}\right),
	\end{equation}
where the atoms $\left(m_i^{(d_{j_1,...,j_n})}, l_i^{(d_{j_1,...,j_n})}\right)$ are defined by~\eqref{eq4.9:}--\eqref{20.4.12}.

 \item  $d_{j_1,...,j_n}$-diagonal full problem  $\bold{MP}(\mathbf{s})$  is indeterminate if and only if 
\begin{equation}\label{20.4.29}
	\sum_{i=1}^\infty m_i^{(d_{j_1,...,j_n})}(0,...,0)<\infty\quad\mbox{and}\quad \sum_{i=1}^\infty l_i^{(d_{j_1,...,j_n})}<\infty.
\end{equation}

\end{enumerate}

Furthermore, if \eqref{20.4.29} holds, then the sequence of resolvents matrices $W_{2N,d_{j_1,...,j_n}}$ converges to the  matrix valued function 
\begin{equation}\label{20.4.19x_1x2}\begin{split}
W^{+}_{\infty, d_{j_1,...,j_n}}(z_1,...,z_n)&=A_{d_{j_1,...,j_n}}(z_1,...,z_n)\times\\&\times\begin{pmatrix}
w^+_{11,d_{j_1,...,j_n}}(z_1,...,z_n)& w^+_{12,d_{j_1,...,j_n}}(z_1,...,z_n)\\
w^+_{21,d_{j_1,...,j_n}}(z_1,...,z_n)&w^+_{22,d_{j_1,...,j_n}}(z_1,...,z_n)
\end{pmatrix},\end{split}
\end{equation}
where the matrix $A_{d_{j_1,...,j_n}}$ is defined by \eqref{20.3.prp4.5.3}
 and $F_{d_{j_1,...,j_n}}$  can be represented as follows
\begin{equation}\label{20.4.19x_1}
	F_{d_{j_1,...,j_n}}(z_1,...z_n)\!=\!\!\cfrac{w^+_{11,d_{j_1,...,j_n}}(z_1,...,z_n)\tau(z_1,...,z_n)+w^+_{12,d_{j_1,...,j_n}}(z_1,...,z_n)}{\prod\limits_{i=1}^n \!\!z_i^{j_i}\!\!\left(\!w^+_{21,d_{j_1,...,j_n}}(z_1\!,\!...\!,\!z_n)\tau(z_1\!,\!...\!,\!z_n)\!+\!w^+_{22,d_{j_1,...,j_n}}(z_1\!,\!...\!,\!z_n)\!\right)},
	\end{equation}
where the parameter $\tau(z_1, ..., z_n)=o\left(1\right)$.
\end{theorem}
\begin{proof} Let the assumption of Theorem hold. Any solution $F_{d_{j_1,...,j_n}}$ of  $d_{j_1,...,j_n}$-diagonal full problem $\bold{MP}(\mathbf{s})$ admits asymptotic expansion \eqref{Fx1} and   $F_{d_{j_1,...,j_n}}$  can be represented by
\begin{equation}\label{Ox1}
	F_{d_{j_1,...,j_n}}(z_1,...,z_n)=\cfrac{1}{\prod\limits_{i=1}^n \!z_i^{j_i}}\tilde{F}_{d_{j_1,...,j_n}}(z_1,...,z_n),
\end{equation}
where $\tilde{F}_{d_{j_1,...,j_n}}$  are defined by \eqref{Fx2} and by Theorem~\ref{20.th3.6} 
\begin{equation}\label{Ox2}
 \tilde{F}_{d_{j_1,...,j_n}}\!\!(\!z_1\!,...,\!z_n\!)\!=\!-\!\!\sum\limits_{k=1}^{\infty}\!\!\cfrac{\left(\!\mathfrak{s}^{(d_{j_1\!,\!...\!,\!j_n})}_{0}\!\right)^k}{\prod\limits_{i=1}^n \!\!z^k_i}\underset{i=1}{\overset{\infty}{\mathbf{K}}}\!\!\left(\!\!\cfrac{1}{-m_i^{(d_{j_1,...,j_n})}\!\!(z_1,..., z_n)\!\!\!\prod\limits_{j=1}^n\!\!z_j+\cfrac{1}{l_i^{(d_{j_1,...,j_n})}}}\!\!\right)\!\!,
\end{equation}
where the atoms $\left(m_i^{(d_{j_1,...,j_n})}, l_i^{(d_{j_1,...,j_n})}\right)$ are defined by~\eqref{eq4.9:}--\eqref{20.4.12} and \eqref{20.4.29} holds. Moreover,  the sequences of resolvent matrices $\tilde{W}_{2N,d_{j_1,...,j_n}}$ converges to 
\begin{equation}\label{Ox3}
	\tilde{W}^{+}_{\infty, d_{j_1,...,j_n}}(z_1,...,z_n)=\times\begin{pmatrix}
w^+_{11,d_{j_1,...,j_n}}(z_1,...,z_n)& w^+_{12,d_{j_1,...,j_n}}(z_1,...,z_n)\\
w^+_{21,d_{j_1,...,j_n}}(z_1,...,z_n)&w^+_{22,d_{j_1,...,j_n}}(z_1,...,z_n)
\end{pmatrix}
\end{equation}
and $\tilde{F}_{d_{j_1,...,j_n}}$  can be represented as
\begin{equation}\label{Ox4}
	\tilde{F}_{d_{j_1,...,j_n}}(z_1,...z_n)\!=\!\!\cfrac{w^+_{11,d_{j_1,...,j_n}}(z_1,...,z_n)\tau(z_1,...,z_n)+w^+_{12,d_{j_1,...,j_n}}(z_1,...,z_n)}{w^+_{21,d_{j_1,...,j_n}}(z_1,...,z_n)\tau(z_1,...,z_n)+w^+_{22,d_{j_1,...,j_n}}(z_1,...,z_n)},
	\end{equation}
where the parameter $\tau(z_1, ..., z_n)=o\left(1\right)$.

According to \eqref{Ox1} and  \eqref{Ox2}, we  obtain  \eqref{20.4.289}--\eqref{20.4.29}. Due to  $${W}_{2N,d_{j_1,...,j_n}}=  A_{d_{j_1,...,j_n}}\tilde{W}_{2N,d_{j_1,...,j_n}},$$ where $A_{d_{j_1,...,j_n}}$ is defined by~\eqref{Ax1},  and  \eqref{Ox3}-- \eqref{Ox4}, we get \eqref{20.4.19x_1x2}--\eqref{20.4.19x_1}.  This completes the proof.~\end{proof}

\section{Full problem in general case}
In the current section, we study the full problem. The results  are based on the Sections 2 and 3, i.e. we apply the diagonal approach to the complete problem.

\begin{proposition}\label{20pxprop5.2} Let  $\mathbf{s}=\{s_{i_1,i_2,...,i_n}\}_{i_1,...,i_n=0}^{\infty}$ be a sequence of real numbers 
and let  all $d_{j_1,...,j_n}$-associated sequences $\mathfrak{s}^{(d_{j_1,...,j_n})}=\left\{\mathfrak{s}^{(d_{j_1,...,j_n})}_{i}\right\}_{i=0}^{\infty}$ be regular and let  all $l_j^{(d_{j_1,...,j_n})}>0$.
Then:
   \begin{enumerate}
 \item Any solution of the  full problem $\bold{MP}(\mathbf{s})$ admits the representation
\begin{equation}\label{20.5.eert_1}\begin{split}
	F(z_1,...z_n)&=\sum^{\infty}_{\begin{matrix} 
 j_i=0,  \\
\mbox{without }j_1\\
   \end{matrix} } \!\!\!\!\! \!\!\!\!\!F_{d_{0,j_2,...,j_n}}(z_1,...z_n)+\sum^{\infty}_{\begin{matrix} 
 j_1=1,j_i=0  \\
\mbox{without }j_2\\
   \end{matrix} }\!\!\!\!\! \!\!\!\!\!\!\!\!F_{d_{j_1,0,j_3,...,j_n}}(z_1,...z_n)+\ldots+\\&+
   \sum^{\infty}_{\begin{matrix} 
 j_i=1  \\
\mbox{without }j_n\\
   \end{matrix} } \!\!\!\!\! \!\!\!\!\!F_{d_{j_1,,...,j_{n-1},0}}(z_1,...z_n),
		\end{split}\end{equation}
	
where the atoms $\left(a_i^{(d_{j_1,...,j_n})}, b_i^{(d_{j_1,...,j_n})}\right)$ and $F_{d_{j_1,j_2,...,j_n}}$ are defined by~\eqref{20.4.6} and \eqref{20.4.eert_1}, respectively.

 \item The  full problem  $\bold{MP}(\mathbf{s})$  is indeterminate if and only if  all $d_{j_1,...,j_n}$-diagonal full problems  $\bold{MP}(\mathbf{s})$  are  indeterminate, i.e. the following holds
 \begin{equation}\label{20.4.eert_2}
 	\sum_{i=0}^{\infty}\!\!\cfrac{\left|P_{i}^{(d_{j_1,...,j_n})}(0,...,0)\right|^2}{\prod\limits_{j=0}^i b_i^{(d_{j_1,...,j_n})}}<\infty\quad\mbox{and}\quad 
	\sum_{i=0}^{\infty}\!\!\cfrac{\left|Q_{i}^{(d_{j_1,...,j_n})}(0,...,0)\right|^2}{\prod\limits_{j=0}^i b_i^{(d_{j_1,...,j_n})}}<\infty.
 \end{equation}
\end{enumerate}

\end{proposition}
\begin{proof} Let the assumptions of statement hold. The sequence $\mathbf{s}=\{s_{i_1,i_2,...,i_n}\}_{i_1,...,i_n=0}^{\infty}$  is associated with the sequence $ \mathfrak{s}$, which can be represented as the union  of $d_{j_1,...,j_n}$-associated sequences $\mathfrak{s}^{(d_{j_1,...,j_n})}$
 \begin{equation}\label{20.4.eert_2cx}
	 \mathfrak{s}=\bigcup_{j_1,...,j_n=0}^{\infty} \mathfrak{s}^{(d_{j_1,...,j_n})}.
 \end{equation}
By Proposition~\ref{20pxprop3.2} and Proposition~\ref{20pxprop4.2}, 
all associated functions  $F_{d_{j_1,...,j_n}}$ can be represented by \eqref{20.4.eert_1} and  any solution of the  full problem $\bold{MP}(\mathbf{s})$ takes the form \eqref{20.5.eert_1}.

Moreover, the  full problem $\bold{MP}(\mathbf{s})$ is the union of all  $d_{j_1,...,j_n}$-diagonal  problem $\bold{MP}(\mathbf{s})$. Hence, the  full problem $\bold{MP}(\mathbf{s})$  is indeterminate iff  all $d_{j_1,...,j_n}$-diagonal full problems  $\bold{MP}(\mathbf{s})$  are  indeterminate,  i.e. \eqref{20.4.eert_2} holds.  This completes the proof.~\end{proof}

\begin{theorem}\label{20.th5.11} Let  $\mathbf{s}=\{s_{i_1,i_2,...,i_n}\}_{i_1,...,i_n=0}^{\infty}$ be the sequence of real numbers 
and let the all $d_{j_1,...,j_n}$-associated sequences $\mathfrak{s}^{(d_{j_1,...,j_n})}=\left\{\mathfrak{s}^{(d_{j_1,...,j_n})}_{i}\right\}_{i=0}^{\infty}$ be regular. Then:
   \begin{enumerate}
 \item Any solution of the  full problem $\bold{MP}(\mathbf{s})$ admits the following  representation
\begin{equation}\label{20.5.289s}\begin{split}
	F(z_1,...z_n)&=\sum^{\infty}_{\begin{matrix} 
 j_i=0,  \\
\mbox{without }j_1\\
   \end{matrix} }\!\!\!\!\! \!\!\!\!\! F_{d_{0,j_2,...,j_n}}(z_1,...z_n)+\sum^{\infty}_{\begin{matrix} 
 j_1=1,j_i=0  \\
\mbox{without }j_2\\
   \end{matrix} } \!\!\!\!\! \!\!\!\!\!F_{d_{j_1,0,j_3,...,j_n}}(z_1,...z_n)+\ldots+\\&+
   \sum^{\infty}_{\begin{matrix} 
 j_i=1  \\
\mbox{without }j_n\\
   \end{matrix} }\!\!\!\!\! \!\!\!\!\! F_{d_{j_1,,...,j_{n-1},0}}(z_1,...z_n),
	\end{split}\end{equation}
where  the atoms $\left(m_i^{(d_{j_1,...,j_n})}, l_i^{(d_{j_1,...,j_n})}\right)$  and  $F_{d_{j_1,...,j_n}}$  are defined by~\eqref{eq4.9:}--\eqref{20.4.12} and \eqref{20.4.289}, respectively .

 \item The  full problem  $\bold{MP}(\mathbf{s})$  is indeterminate if and only if  all $d_{j_1,...,j_n}$-diagonal full problems  $\bold{MP}(\mathbf{s})$  are  indeterminate , i.e. the following holds
\begin{equation}\label{20.5.29}
	\sum_{i=1}^\infty m_i^{(d_{j_1,...,j_n})}(0,...,0)<\infty\quad\mbox{and}\quad \sum_{i=1}^\infty l_i^{(d_{j_1,...,j_n})}<\infty.
\end{equation}

\end{enumerate}

\end{theorem}
\begin{proof} Let the assumptions of Theorem hold. The sequence $\mathbf{s}=\{s_{i_1,i_2,...,i_n}\}_{i_1,...,i_n=0}^{\infty}$ is associated with the sequence $ \mathfrak{s}$, which is defined by  \eqref{20.4.eert_2cx}. Therefore,  the  full problem $\bold{MP}(\mathbf{s})$ is the union of all  $d_{j_1,...,j_n}$-diagonal  problem $\bold{MP}(\mathbf{s})$ and the solution of the full problem $\bold{MP}(\mathbf{s})$  can be represented by \eqref{20.5.289s} and the full problem  $\bold{MP}(\mathbf{s})$  is indeterminate iff  all $d_{j_1,...,j_n}$-diagonal full problems  $\bold{MP}(\mathbf{s})$  are  indeterminate, i.e. \eqref{20.5.29} holds. This completes the proof.~\end{proof}


\end{document}